\documentclass[11pt, a4paper]{amsart}
\usepackage{amsmath, amsthm, amssymb, amsfonts, enumerate}
\usepackage[colorlinks=true,linkcolor=blue,urlcolor=blue]{hyperref}
\usepackage{verbatim}
\usepackage[dvips]{color}
\usepackage[active]{srcltx}
\usepackage{mathrsfs}
\usepackage{latexsym}
\usepackage{graphicx}
\usepackage{float}
\usepackage{mathrsfs}
\usepackage[active]{srcltx}
\usepackage{xcolor}
\usepackage{geometry}\usepackage{epstopdf}
\usepackage{commath}
\usepackage{caption}
\usepackage{accents}
\usepackage{soul}
\usepackage{bbm}
\usepackage{bm}

\usepackage{mathtools}
\mathtoolsset{showonlyrefs}

\allowdisplaybreaks

\newtheorem{theorem}{Theorem}[section]
\newtheorem{assumption}[theorem]{Assumption}
\newtheorem{proposition}[theorem]{Proposition}

\newtheorem{corollary}[theorem]{Corollary}
\theoremstyle{definition}

\newtheorem{remark}[theorem]{Remark}
\newtheorem{example}{Example}

\floatstyle{plain} \restylefloat{figure} \restylefloat{table}

\def\A{\mathcal A}

\def\C{\mathcal C}

\def\E{\mathbb E}
\def\F{\mathcal F}
\def\G{\mathcal G}

\def\P{\mathbb P}

\def\R{\mathbb R}

\def\X{\mathcal X}

\def\ud{\mathrm d}

\newcommand\bE{\mathbb{E}}
\newcommand\bF{\mathbb{F}}

\newcommand\bP{\mathbb{P}}

\newcommand{\ind}{\mathbbm{1}}

\newcommand{\kom}[1]{}
\renewcommand{\kom}[1]{{\bf [#1]}}
\definecolor{blau}{rgb}{0.1,0.0,0.9}

\newcounter{komcounter}
\numberwithin{komcounter}{section}

\title[Irreversible investment with learning-by-doing]{An irreversible investment problem with a learning-by-doing feature}
\author[E.\ Ekstr\"om, Y. Kitapbayev, A. Milazzo, T. Tolonen-Weckstr\"om]{Erik Ekstr\"om, Yerkin Kitapbayev, Alessandro Milazzo and Topias Tolonen-Weckstr\"om}
\address{E.\ Ekstr\"om and T.\ Tolonen-Weckstr\"om: Department of Mathematics, Uppsala University, Box 256, 75105 Uppsala, Sweden.}

\email{\href{mailto: ekstrom@math.uu.se}{ekstrom@math.uu.se}}
\email{\href{mailto: topias.tolonen@math.uu.se}{topias.tolonen@math.uu.se}}

\address{Y. Kitapbayev: Mathematics Department,
Khalifa University of Science and Technology, PO Box 127788, Abu Dhabi, United Arab Emirates.}

\email{\href{mailto: yerkin.kitapbayev@ku.ac.ae}{yerkin.kitapbayev@ku.ac.ae}}

\address{A. Milazzo: School of Management and Economics, Dept.~ESOMAS, University of Turin, C.so Unione Sovietica 218bis, 10134 Turin, Italy.}
\email{\href{mailto: alessandro.milazzo@unito.it}{alessandro.milazzo@unito.it}}

\date{\today}

\begin{document}

\begin{abstract}
We study a model of irreversible investment for a decision-maker who has the possibility to gradually invest in a project with unknown value. In this setting, we introduce and explore a feature of ``learning-by-doing'', where the learning rate of the unknown project value is increasing in the decision-maker's level of investment in the project.
We show that, under some conditions on the functional dependence of the learning rate on the level of investment (the ``signal-to-noise ratio''), the optimal strategy is to invest gradually in the project so that a two-dimensional sufficient statistic reflects below a monotone boundary. Moreover, this boundary is characterized as the solution of a differential problem. Finally, we also formulate and solve a discrete version of the problem, which mirrors and complements the continuous version.
\end{abstract}

\maketitle

\section{Introduction}\label{sec:Intro}
Consider a decision-maker who aims to invest optimally in a new project, but
who suffers from incomplete information and does not know
the true value of the project.
In the simplest Bayesian setting, we assume that the project value $\mu$ has a two-point prior distribution. Thus, it can take two values: $\mu_0<0$ representing a ``bad'' project and $\mu_1>0$ representing a ``good'' project. 
In addition to the prior knowledge of the distribution of $\mu$, the decision-maker has also access to a stream of noisy observations of
the unknown project value $\mu$, and can thus make further inference about its true value. 
%More specifically, we will assume that the observation process is a Brownian motion with unknown drift $\mu$.
Within this set-up with incomplete information, we consider a situation in which the actions of the 
decision-maker may affect the learning rate of the unknown value.
More specifically, we introduce and study 
a notion of {\bf learning-by-doing}: by investing more into the project (i.e., by {\bf doing}) the decision-maker can  improve the {\bf learning} rate of the true value of $\mu$. The decision to increase the level of investment is irreversible, however, and there is thus a natural trade-off between early investment to increase the learning rate and a more cautious strategy to avoid investing in a potentially bad project.

We model the above situation with a learning-by-doing feature by introducing an observation process $X=(X_t)_{t\geq 0}$ of the form\footnote{
Notice that the parameterization $\theta=\theta(\mu):=(\mu-\mu_1)/(\mu_1-\mu_0)$ in the drift of $X$ is chosen, without loss of generality, so that $\theta\in\{0,1\}$.}
% \footnote{{\color{red}Notice that the parameterization $\theta=\theta(\mu):=(\mu-\mu_1)/(\mu_1-\mu_0)$ in the drift of $X$ was chosen only to have $\theta\in\{0,1\}$, instead of $\mu\in\{\mu_0,\mu_1\}$, and will be used in the rigorous formulation of Section \ref{sec3}.}}
\begin{equation}\label{eq:startX}
\ud X_t =\frac{\mu-\mu_0}{\mu_1-\mu_0}\rho(U_t)\ud t+\ud W_t,
\end{equation}
where $W=(W_t)_{t\geq 0}$ is a standard Brownian motion, $U=(U_t)_{t\geq 0}$ is a non-decreasing control process with $0\leq U\leq 1$ that describes the level of investment in the project, and $u\mapsto \rho(u)$ 
(the signal-to-noise ratio) is a given positive and increasing function (for a more precise description of $\rho$ and the set of admissible control processes, see Section~\ref{sec3} below). That is, the decision-maker observes the process $X$ and may choose to increase the level of investment $U$ at any time, thereby 
increasing the signal-to-noise ratio of the observation process and thus obtaining a 
faster estimation of the true value $\mu$ of the project.
In this setting, the objective for the decision-maker is 
to choose a non-decreasing control $U$ 
to maximize the expectation 
\begin{equation}
\label{formulation}
\bE\left[\int_0^\infty e^{-rt}\mu \,\ud U_t\right]
\end{equation}
of the accumulated discounted true value of future investments, where $r>0$ is a subjective discount rate. Note that 
the optimization of \eqref{formulation} over investment strategies, subject to 
the learning-by-doing feature as described in \eqref{eq:startX}, results in an intrinsic ``cost of learning'': the decision-maker naturally wants to improve his/her learning rate of $\mu$ by increasing the control $U$, 
but, by doing so, she may be investing in a bad project (with $\mu=\mu_0<0$).
Our problem formulation of irreversible investment with learning-by-doing thus describes an instance of the classical theme of {\bf exploration} vs. {\bf exploitation}.

Intuitively, the set-up with the feature of learning-by-doing described above can be motivated as follows. An ``outsider'' (an agent who is not invested at all) may have access to noisy observations of the true value $\mu$ of a certain 
project, for example by observing financial statements of companies that are currently operating in a similar line of business. With no -- or little -- involvement in the project, however, the outsider 
has to rely on publicly available information, and observations of the project value are rather noisy.
On the other hand, with a greater involvement in the project, as measured by the decision-maker's investment level, additional private information becomes available and more
precise inference of the intrinsic project value can be obtained.

For instance,
an improved learning rate is a natural ingredient in situations allowing for {\bf project expansion} (see Example \ref{ex:ProjExp} for a specific formulation). 
%Indeed, consider a decision-maker who runs a small scale project, e.g., a small business in a specific line of trade.
% The true project value is not known due to uncertainties in, for example, production costs and unknown demand, and one merely has access to a stream of noisy observations of the project value. 
% If the business is expanded, e.g., by starting another identical business operating in the same field with the same unknown project value, then learning of the true project value is enhanced by a two-source learning procedure.
As an example, consider the renewable energy sector, where a firm invested in wind turbines or solar cell plants
may encounter uncertainties related to factors such as weather patterns and energy output, as well as wear and tear and maintenance cost. 
Consequently, the firm  has only access to a noisy stream of observations of the true project value.
By gradual project expansion, however, some of the uncertainty factors are observed with more precision thanks to a higher experimentation rate, which enables the firm to better assess the viability of future expansion.
A second example involves the launch of a new product or an existing product into a new market, where 
the true project value is not known due to uncertainties in, for example, production costs and  demand. Again, project expansion gives rise to a higher experimentation rate, which leads to less noisy observations of the project value.

The learning-by-doing mechanism may also be associated with the level of {\bf commitment}. 
For example, consider an agent who may invest in a new start-up, whose profitability is uncertain. 
By increasing the investment level, the agent shows commitment to the start-up and may to a larger extent gain access to board meetings and other events where more information is revealed, thereby reducing the level of noise in the observations of the project value. We also use this example to better illustrate our objective function~\eqref{formulation} and the role of discounting. Suppose the investor purchases additional shares of the start-up, whose intrinsic value $\mu$ takes two values $\mu_0<0$ (bad project) and $\mu_1>0$ (good project). For instance, for a start-up developing a cutting-edge technology, it is plausible that $\mu=\mu_1$ with high probability. If, at time $t>0$, the investor increases his/her investment by $\ud U_t$ then the (pure) value of that investment today is $e^{-rt}\,\ud U_t$. By additionally taking into account the intrinsic value of the start-up, the present true value of the investment is $\mu e^{-rt}\,\ud U_t$. Taking expectation -- since the investor cannot directly observe $\mu$ -- and aggregating the investments by integration naturally leads to our form of the objective function~\eqref{formulation}.

Note that in \eqref{formulation} above, no investment costs are included. This, however, is without loss of generality. Indeed, if a constant investment cost $C>0$ is included in the model, then the expected total discounted value of investment would be
\begin{equation}
    \label{investmentcost}
    \E\left[\int_0^\infty e^{-rt}( \mu-C)\,\ud U_t\right] .
    \end{equation}
Consequently, the optimization over controls $U$ of the expression in \eqref{investmentcost} is of the same type as in \eqref{formulation}, but with $\mu$ replaced by $\tilde\mu:=\mu-C$ (for a non-degenerate problem one then needs 
$\mu_0<C<\mu_1$).
% {\color{red}
%     Also note that we in \eqref{formulation} above consider an investment problem where the payoff consists of a non-observable and abstract project value $\mu$. For an example with observable cash flows (that is contained within our problem formulation), see Example~\ref{subsection:cashflows} in Section~\ref{sec7}.
% }
Also note that $\mu$ is not directly observable, but represents the intrinsic value of the project. 
In many classical references on irreversible investment
(cf.\ \cite{DP}, \cite{MS}, and also \cite{GN} for a more recent contribution with incomplete information), 
the project value is inferred as the present value of cumulative future revenues from an investment; our set-up is more general with 
$\mu$ being an abstract intrinsic value, but we note that the setting also admits an interpretation as the present value of cumulative future revenues from an investment (cf.\ Example~\ref{subsection:cashflows}).

% \textcolor{red}{
% With this view in mind, we mention three possible applications that can suit this framework, among potentially many more. First, an agent who wants to invest in a new start-up, whose potential profit is unknown and can be better estimated by increasing the investment level and thus, e.g., gaining the access to board meetings where more information is revealed. %Second, the investment in a new technology, such as power plants for renewable energy, whose real benefit can be better understood only by practically investing more and more resources.
% The second example involves the launch of a new product or an existing product into a completely new market. In such scenarios, gradual investment is essential to uncover uncertainties surrounding demand. The third example is the renewable energy sector, where firms considering  novel solar or wind energy projects often encounter uncertainties related to factors such as weather patterns, government policies, and technological advancements. For instance, a company investing in a solar farm may initially face uncertainties regarding solar radiation, energy output, maintenance costs, and regulatory challenges. The learning-by-doing process will enable renewable energy firms to better assess the viability of future expansion.}

\subsection{Related literature}

Problems of irreversible investment have been widely studied in the literature on stochastic control, with early references provided by 
\cite{DP} and \cite{MS}, and more recent contributions including, among many others, 
\cite{bank2020modelling},
%\cite{chiarolla2013generalized}, 
%\cite{baldursson1996irreversible}, 
\cite{de2017optimal}, 
\cite{decamps2006irreversible} and
%\cite{chiarolla2005explicit}, 
\cite{ferrari2015integral}.
Mathematically, the irreversible investment problem described in \eqref{eq:startX}--\eqref{formulation} (and formulated more precisely in Section~\ref{sec3} below) is a stochastic singular control problem under incomplete information, where the chosen control affects the learning rate. While stochastic control problems with incomplete information have been studied extensively 
(for early references, see \cite{L} for a problem of utility maximization, and \cite{DMV} for an investment timing decision), works involving control of the learning rate were previously more rare but have seen a growing number of contributions in recent years. 
Within statistics, a problem of change-point detection with a controllable learning rate has been studied in
\cite{DS} (for reversible controls) and \cite{EM2023} (for irreversible controls), and estimation problems with a controllable learning rate were considered in \cite{moscarini2001optimal}, \cite{EK} and \cite{campbell2025sequential}.
In literature on operations management, related questions of the trade-off between earning and learning have been studied in the context of dynamic pricing in models with demand uncertainty, see, e.g., \cite{HKZ}.
Within operations research, \cite{harrison2015investment} studies an investment problem similar to ours, but with the main difference that
investment does not affect the learning rate.
More specifically, an investor may at each instant in time choose between a finite set of learning rates, where a larger learning rate comes with a larger running cost of observation, and where the unknown return has a two-point distribution. Moreover, the investor may choose an investment time, at which the optimization ends.  For a related work, see also 
\cite{sunar2021competitive}. Some abstract stochastic control problems where controls affect the learning rate have been recently studied also in \cite{cohen2025optimal}, \cite{knochenhauer2024continuous}, and \cite{cox2025measure}.
Finally, our problem is also related to classical multi-armed bandit problems, where a chosen strategy 
affects both learning and earning; see e.g. \cite{GGW}.

We also remark that, with respect to most of the existing literature of stochastic control problems where the control affects the learning rate (see, e.g., \cite{DS}, \cite{EK}, \cite{EM2023}, \cite{campbell2025sequential}, \cite{cohen2025optimal} and \cite{knochenhauer2024continuous}), we do not fix any specific form of the dependence of the learning rate on the control (the so-called ``signal-to-noise ratio'' in our problem). Instead, we develop and study a more flexible set-up by allowing for an arbitrary signal-to-noise ratio and providing sufficient conditions that guarantee the existence of a solution, which we can explicitly describe.

Since reversible controls are considered in \cite{DS}, \cite{EK} and \cite{harrison2015investment}, the sufficient statistic in those studies consists merely of the conditional probability of one of the two possible states, and is thus one-dimensional. On the other hand, for irreversible controls (as in \cite{EM2023}, and in the current paper), the sufficient statistic consists of the conditional probability of one of the states together with the current value of the control, and is thus two-dimensional. 

The epithet ``learning-by-doing'' has been associated to various problems in the economics literature, mainly in settings where an experienced agent has a larger ability than a less experienced one; for a classical reference, see \cite{arrow1962economic}.
One may note that the notion of ``learning-by-doing'' as used in \cite{arrow1962economic} could alternatively be described as ``improving-by-doing'', whereas the notion of the present paper could 
alternatively be referred to as ``learning-faster-by-doing''. Indeed, in \cite{arrow1962economic} the {\em profitability} rate is larger for an experienced agent. 
In contrast, for us the investment level does not influence the project value $\mu$, but it affects instead the rate with which the project value is revealed to the agent.

\subsection{Preview}

The remainder of the paper is organized as follows. 
% In Section~\ref{sec2} we discuss a few aspects of the irreversible investment problem under consideration {\color{red} in an informal manner, providing an intuitive interpretation of the problem we wish to model and some of its features.} %, e.g., 
%a particular case in which the agent 
%is able to filter out systematic noise at a rate that increases with the level of investment, and the inclusion of investment costs.
Section~\ref{sec3} offers a precise mathematical formulation of the problem, using a weak approach. In Section~\ref{sec4}, we provide heuristic reasoning to derive an ordinary differential equation (ODE) for a boundary, along which a candidate optimal strategy reflects the underlying sufficient statistic, and Section~\ref{sec5} discusses conditions under which the solution of the ODE is monotone increasing. 
In Section~\ref{sec6}, we provide a verification theorem, which guarantees that the obtained candidate strategy is indeed optimal. In Section~\ref{sec7}, we provide some specific examples for our model with the corresponding illustrations of the optimal boundary. Finally, in Section~\ref{sec8}, we analyze a discrete version of our problem, which corresponds to situations where the set of possible investment levels is discrete.

\section{Problem formulation}\label{sec3}

In the Introduction, it is implicitly understood that the control $U$ should be chosen based on available observations of the process $X$. On the other hand, 
the choice of a control $U$ affects the observation process $X$, cf.\ \eqref{eq:startX}.
Because of this circular interplay, special care is needed when describing the set of admissible controls. Problems of this type are well-suited for 
the ``weak formulation'' based on change of measures and the Girsanov theorem (see also \cite{EK} and, for a similar situation as we have with irreversible controls, \cite{EM2023}). This approach is introduced in the current section.

% {\color{red}
% In this section we provide a weak formulation of our investment problem based on change of measures. 
% }
Let $(\Omega,\mathcal F,\P)$ be a complete probability space, supporting a standard Brownian motion $X$ and an independent Bernoulli random variable $\theta$ with 
\[\P(\theta=1)=\pi=1-\P(\theta=0), \qquad \pi\in(0,1).\]
Let $\mathbb F=(\F_t)_{t\geq 0}$ be the smallest right-continuous filtration to which the process $X$ is adapted, and 
$\mathbb G=(\G_t)_{t\geq 0}$ the smallest right-continuous filtration to which the pair $(X,\theta)$ is adapted.
Denote by $\A$ the collection of $\bF$-adapted, right-continuous, non-decreasing processes $U$ with values in $[0,1]$; for $u\in[0,1]$, denote by $\A_u$ the sub-collection of controls with initial value equal to $u$, i.e.,
\begin{equation}\label{A_u}
\A_u=\{U\in\A: U_{0-}=u\}.
\end{equation}
Let $\rho:[0,1]\to(0,\infty)$ be a given non-decreasing and bounded function. Then,
for any $U\in\A$ and $t\in[0,\infty)$, we can define a measure $\P^U_t\sim \P$ on $(\Omega,\G_t)$ by
\[\frac{\ud\P_t^U}{\ud\P}:=\exp\left\{\theta\int_0^t \rho(U_s)\,\ud X_s-\frac{\theta^2}{2}\int_0^t\rho^2(U_s)\,\ud s\right\}=:\eta^U_t.\]
Setting $\G_\infty:=\sigma (\cup_{0\leq t<\infty}\G_t)$, we may assume the existence of a probability measure $\P^U$ on $(\Omega, \G_\infty)$ that coincides with $\P^U_t$ on $\G_t$ (this can be guaranteed, e.g., by the theory of the so-called F\"{o}llmer measure, cf.\ \cite{follmer1972exit}).
By the Girsanov theorem, 
\[W^U_t:=X_t-\theta\int_0^t\rho(U_s)\,\ud s\]
is a $(\P^U,\mathbb G)$-Brownian motion, and consequently $X$ has the representation
\[X_t=\theta\int_0^t\rho(U_s)\,\ud s + W^U_t.\]
Note that this coincides with the dynamics of the observation process described in \eqref{eq:startX}, where $\theta=(\mu-\mu_0)/(\mu_1-\mu_0)$, and that the function $\rho(\cdot)$ is the signal-to-noise ratio of the problem.

It should be noticed that the law of $\theta$ remains the same under $\bP^U$ as under $\bP$. Indeed, denoting by $\bE^U$ the expectation under $\bP^U$, we have
$$\bP^U(\theta=1)=\bE^U[\ind_{\{\theta=1 \}}]=\bE[\ind_{\{\theta=1 \}}\eta^U_0]=\bP(\theta=1)=\pi,$$
where the second equality follows from the fact that $\theta$ is $\G_0$-measurable.

For any $U\in\A$, define the adjusted belief process
\[\Pi^U_t:=\P^U(\theta=1\vert \F_t), \qquad t\in[0,\infty).\]
By the innovations approach to stochastic filtering, the so-called {\em innovations process}
\[\hat{W}^U_t:=X_t-\int_0^t\rho(U_s)\Pi^U_s\,\ud s\]
is a $(\P^U,\mathbb F)$-Brownian motion, and (see, e.g., \cite[Theorem 8.1]{liptser1977statistics})
\begin{equation}\label{PiSDE}
\ud\Pi^U_t=\rho(U_t)\Pi^U_t(1-\Pi^U_t)\,\ud \hat{W}^U_t.
\end{equation}
Notice that, since $\rho$ is bounded by assumption, the SDE \eqref{PiSDE} admits a unique strong solution.

% To embed our problem in a Markovian framework, for any $U\in\A$, we define the adjusted belief process
% \[\Pi^U_t:=\P^U_\pi(\theta=1\vert \F_t), \qquad t\in[0,\infty),\]
% where the sub-index is used to indicate the initial belief $\Pi_0=\pi$. 

Recall the original problem formulation \eqref{formulation}. We note that conditioning yields 
\begin{align*}%\label{eq:PiFormulation}
   \bE^U\bigg[\int_0^\infty e^{-rt}\mu \,\ud U_t\bigg] &=\bE^U\bigg[\int_0^\infty e^{-rt}\bE^U\big[\mu\big|\F_t\big] \ud U_t\bigg]\\
  % &=\int_0^\infty e^{-rt}\bE^U\big[U_t(\mu_0(1-\Pi^U_t)+ \mu_1\Pi^U_t )\big]\ud t \nonumber\\
   &=(\mu_1-\mu_0)\bE^U\bigg[\int_0^\infty e^{-rt}(\Pi^U_t-k)\,\ud U_t \bigg],
\end{align*}
where $k:=-\mu_0/(\mu_1-\mu_0)\in (0,1)$
and the integral is interpreted in the Riemann-Stieltjes sense over the interval $[0,\infty)$. %i.e. including possible jumps of $U$ at time 0. 
We thus see that 
\[\sup_{U\in\A} \bE^U\bigg[\int_0^\infty e^{-rt}\mu \,\ud U_t\bigg]
= (\mu_1-\mu_0)\sup_{U\in\A} \bE^U\bigg[\int_0^\infty e^{-rt}(\Pi^U_t-k)\,\ud U_t \bigg],\]
and a strategy $U$ that is optimal on the left-hand side will also be optimal on the right-hand side, and vice versa. Therefore, in order to solve our original problem~\eqref{formulation}, we introduce the value function (leaving out the multiplicative factor $\mu_1-\mu_0$)
\begin{equation}\label{V}
V(u,\pi):=\sup_{U\in \A_u}\E^U_\pi\left[\int_0^\infty e^{-rt}(\Pi^U_t-k)\,\ud U_t\right], \qquad (u,\pi)\in[0,1]\times(0,1),
\end{equation}
where $k\in(0,1)$ and the sub-index $\pi$ denotes the expectation $\E^U_\pi[\cdot]:=\E^U[\cdot|\Pi^U_0=\pi]$.%, which indicates the prior probability that $\theta=1$.

\begin{remark}
    An alternative way to treat the circular dependence between the observation process and the controls %(recall Section \ref{sec:AdmContr})
    is through the ``strong formulation''. This is, for instance, the approach in \cite{cohen2025optimal, cox2025measure, knochenhauer2024continuous, campbell2025sequential}. In the strong formulation, one first fixes a Brownian motion $W$ and the filtration it generates, denoted by $\bF^W$. Then, the class of pre-admissible controls $\A^{\text{pre}}_u$ is defined as the collection of $\bF^W$-adapted, right-continuous, non-decreasing processes $U$ with values in $[0,1]$ and such that $U_{0-}=u$. For any $U\in\A^{\text{pre}}_u$, the observation process $X^U$ is the unique strong solution to the SDE
    $$\ud X^U_t=\theta\rho(U_t)\ud t+\ud W_t.$$
    Any such process observation process $X^U$ generates a filtration which is denoted by $\X^U$. The class of admissible controls in the strong formulation is, thus, defined as the set of pre-admissible controls that are also adapted with respect to the observation filtration, i.e.,
    $$\A_u:=\{U\in\A^{\text{pre}}_u: U \text{ is } \X^U\text{-adapted} \}.$$
    This formulation comes with a few issues to be taken care of. In particular, the set of admissible controls in the strong formulation is not closed under addition (see also \cite[Remark 2.6]{cohen2025optimal}).
    Moreover, notice that \cite{cohen2025optimal, cox2025measure, knochenhauer2024continuous, campbell2025sequential} consider ``classical'' controls, whereas our controls are of singular nature. 
\end{remark}

\section{Construction of a candidate solution}\label{sec4}

In this section we use heuristic arguments to construct a candidate value function $\hat V$ and a candidate optimal strategy $\hat U$ for the problem \eqref{V}.
Conditions under which $\hat U$ is optimal are then provided in Section~\ref{sec6} below, along with the equality $V=\hat V$.

It is intuitively clear that one should increase an optimal control $\hat U$ only if $\Pi^{\hat U}$ is large enough. 
Inspired by standard results in singular control, we will construct $\hat V$ using the Ansatz that there exists an non-decreasing boundary $\pi\mapsto h(\pi)$ and, for any initial point $(u,\pi)\in[0,1]\times(0,1)$, an optimal control that
satisfies\footnote{
The existence of a control $\hat U$ that satisfies equation \eqref{heur} is left for now. Note that one cannot simply see \eqref{heur} as a definition, since $\hat U$ appears on both sides of the equation; for a formal definition of $\hat U$, see \eqref{Uhat} below.}
\begin{equation}\label{heur}
\hat U_t=u\vee \sup_{0\leq s\leq t}h(\Pi^{\hat U}_s).
\end{equation}
That is, we postulate that the optimal investment $\hat U$ is gradually increased in such a way that the two-dimensional process $(\hat U,\Pi^{\hat U})$ reflects along the boundary $h$, with reflection in the $u$-direction (see Figure \ref{F:simulation}); if the initial point $(u,\pi)$ satisfies $u<h(\pi)$ then the construction results in an initial jump in the control of size $\ud \hat U_0=h(\pi)-u$.

For any $U\in\A_u$, by the dynamic programming principle, one expects the process 
\[M^U_t:=e^{-rt}\hat V( U_t,\Pi^{ U}_t) + \int^t_0e^{-rs}(\Pi^{ U}_s-k)\,\ud U_s\]
to be a $\P^{ U}$-supermartingale and a $\P^{\hat U}$-martingale when $U=\hat U$. By heuristically applying Itô's formula, the supermartingale condition translates into
\begin{equation}\label{supermg0}
\frac{\rho^2(u)}{2}\pi^2(1-\pi)^2 \hat V_{\pi\pi}-r\hat V \leq 0
\end{equation}
and 
\begin{equation}
\label{supermg}
\hat V_u+\pi-k\leq 0
\end{equation}
at all points $(u,\pi)\in[0,1]^2$, and the martingale condition translates into
\begin{equation}\label{mg}
\frac{\rho^2(u)}{2}\pi^2(1-\pi)^2 \hat V_{\pi\pi}-r\hat V =0
\end{equation}
in the {\em no-action region}
$\C:=\{(u,\pi):u>h(\pi)\}$
and 
\begin{equation}\label{bc}
\hat V_u+\pi-k = 0
\end{equation}
in $[0,1]^2\setminus \C$. The optimality conditions \eqref{supermg0}--\eqref{bc} can be used to derive an additional boundary condition along $\partial \C$. First if all,  
in view of \eqref{supermg} and \eqref{bc}, we should have 
$\hat V_{u\pi}(h(\pi),\pi)+1\geq 0$. 
Also, the Ansatz of an initial jump of size $h(\pi)-u$ for $u\leq h(\pi)$ implies that
\begin{equation}\label{eq:Vjump}
    \hat V(u,\pi)=\hat V(h(\pi),\pi)+(h(\pi)-u)(\pi-k)
\end{equation}
for $u\leq h(\pi)$. The latter together with \eqref{bc} gives
\begin{equation}\label{eq:Vpipi}
    \hat V_{\pi\pi}(u,\pi)=\hat V_{\pi\pi}(h(\pi),\pi) + (\hat V_{u\pi}(h(\pi),\pi)+1)h'(\pi).
\end{equation}
Since $h$ is increasing, by substituting \eqref{eq:Vpipi} and \eqref{eq:Vjump} in \eqref{supermg0} and using \eqref{mg} (which holds at $u=h(\pi)$ by continuity), one obtains that $\hat V_{u\pi}(h(\pi),\pi)+1\leq 0$,
and the additional boundary condition is thus
\begin{equation}\label{bc2}
\hat V_{u\pi}(h(\pi),\pi)+1=0.
\end{equation}
The latter condition on the mixed second derivative of the value function is a well-known recurring condition for two-dimensional singular control problems, see, e.g., \cite{merhi2007model}, \cite{EM2023}, and \cite{de2024maximality}.

Denoting by $b:=h^{-1}$ the inverse of $h$, we thus formulate the following free-boundary problem: find $(\hat V,b)$ such that 
\begin{equation}
\label{fbp}
\left\{\begin{array}{rl}
\frac{\rho^2(u)}{2}\pi^2(1-\pi)^2 \hat V_{\pi\pi}-r\hat V =0 & \pi <b(u)\\
\hat V_u=k-\pi & \pi=b(u)\\
\hat V_{u\pi} =-1 & \pi=b(u)\\
\hat V(u,0+)=0,\end{array}\right.
\end{equation}
where the last condition corresponds to no further investment in the case when the project value $\mu$ is known to be of the negative type.

\begin{figure}[h]
\centering
\includegraphics[scale=0.6]{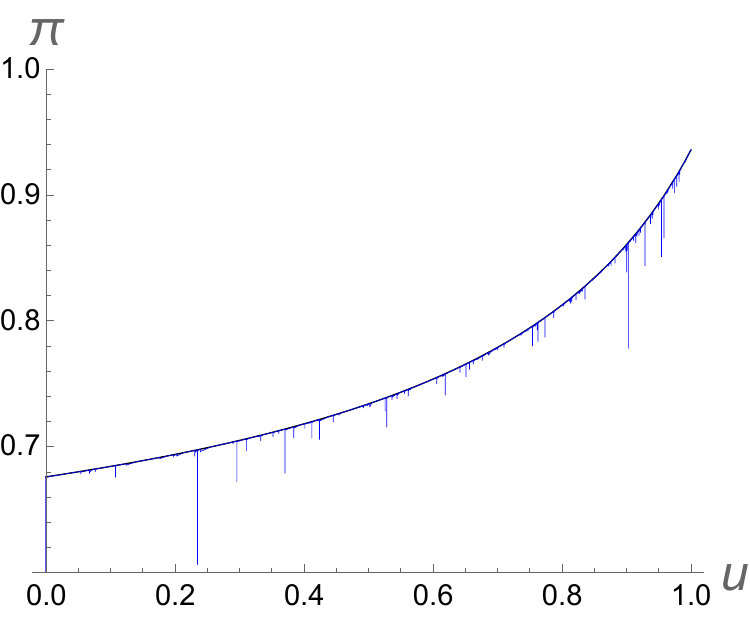}
\caption{The trajectory of the pair $(U,\Pi^U)$ under the reflecting strategy \eqref{heur} in the case 
$\rho^2(u)=\frac{1}{4(1-0.9u)}$, $k=0.5$ and $r = 0.1$.} 
\label{F:simulation}
\end{figure}

\subsection{Deriving an ODE for the free boundary}

The general solution of the ODE in \eqref{fbp} is 
\[\hat V(u,\pi)= A(u)(1-\pi)\left(\frac{\pi}{1-\pi}\right)^{\gamma(u)} + B(u)(1-\pi)\left(\frac{\pi}{1-\pi}\right)^{1-\gamma(u)} ,\]
where $A$ and $B$ are arbitrary functions and $\gamma(u)>1$ is the unique positive solution of the quadratic equation
\begin{equation}\label{gamma}
\gamma^2-\gamma-\frac{2r}{\rho^2(u)}=0.
\end{equation}
More explicitly,
\begin{equation}\label{eq:gamma}
    \gamma(u)=\frac{1}{2}+\sqrt{\frac{1}{4}+\frac{2r}{\rho^2(u)}}.
\end{equation}

Throughout Sections~\ref{sec4}--\ref{sec7} we work under the following assumption.

\begin{assumption}\label{assrho}
The signal-to-noise ratio $\rho:[0,1]\to(0,\infty)$
is twice continuously differentiable, with $\rho'(u)>0$ for every $u\in[0,1]$.
\end{assumption}

\begin{remark}
It is immediate to check that Assumption~\ref{assrho} 
implies that $\gamma:[0,1]\to(1,\infty)$ is twice continuously differentiable, with 
$\gamma'(u)<0$ for every $u\in[0,1]$.
\end{remark}

In view of the boundary condition at $\pi=0+$, we must have $B\equiv 0$ in the Ansatz above.
Introducing the function 
\begin{equation}\label{G}
G(u,\pi):=(1-\pi)\left(\frac{\pi}{1-\pi}\right)^{\gamma(u)},
\end{equation}
our Ansatz then takes the form
\[\hat V(u,\pi)=A(u)G(u,\pi)\] 
for $\pi\leq b(u)$. The two conditions at the boundary (i.e., $\hat V_u=k-\pi$ and $\hat V_{u\pi}=-1$) then become
\[\begin{pmatrix}
G_u(u,b(u)) &  G(u,b(u))\\
G_{u\pi}(u,b(u)) & G_\pi(u,b(u))
\end{pmatrix}\begin{pmatrix}
    A(u) \\
    A'(u)
\end{pmatrix} 
=  \begin{pmatrix}
 k- b(u)\\
 -1\end{pmatrix} ,\]
which yields
\[
\begin{pmatrix}
    A \\
    A'
\end{pmatrix} =
\frac{1}{G_uG_\pi-GG_{u\pi}}
\begin{pmatrix}
G_\pi &  -G\\
-G_{u\pi} & G_u
\end{pmatrix}\begin{pmatrix}
  k-b\\
 -1\end{pmatrix}
\]
(where the arguments of $A=A(u)$, $G=G(u,b(u))$, $b=b(u)$ and their derivatives are omitted).
A straightforward calculation leads to
\[G_uG_\pi-GG_{u\pi}=\frac{-\gamma'}{b(1-b)}G^2,\]
so
\begin{equation}\label{meqn}
\begin{pmatrix}
    A \\
    A' \end{pmatrix} =
\frac{b(1-b)}{\gamma' G^2}
\begin{pmatrix}
G_\pi &  -G\\
-G_{u\pi} & G_u
\end{pmatrix}\begin{pmatrix}
  b-k\\
 1\end{pmatrix}.
\end{equation}
Using 
\[G_\pi=\frac{\gamma-b}{b(1-b)}G,\]
the first equation in \eqref{meqn} simplifies to
\begin{equation}\label{A}
A= \frac{(\gamma+k-1)b-\gamma k}{\gamma'G}.
\end{equation}
Differentiation then gives 
\[A'=\frac{\gamma' ((\gamma+k-1) b'+\gamma' (b-k))G-((\gamma+k-1)b-\gamma k)(\gamma''G+ \gamma'G_u+\gamma'G_\pi b') }{(\gamma')^2 G^2}.\]
Comparing the last equation with the second equation in \eqref{meqn}, we find that 
\begin{eqnarray*}
    &&
\gamma' ((\gamma+k-1) b'+\gamma' (b-k))G-((\gamma+k-1)b-\gamma k)(\gamma''G+ \gamma'G_u+\gamma' b'G_\pi)\\
&&\hspace{15mm}=b(1-b)\gamma'(G_u-(b-k)G_{u\pi}),
\end{eqnarray*}
which simplifies to
\begin{equation}\label{ode}
b'(u)=F(u,b(u)),
\end{equation}
where 
\begin{equation}\label{F}
F(u,b):=\frac{2(b-k)(\gamma')^2+\left(\gamma k-(\gamma+k-1)b\right)\gamma''}
{-\gamma(\gamma k-(\gamma+k-1)b)-(\gamma-1)(1-k)b}\cdot\frac{b(1-b)}{\gamma'}.
\end{equation}

The ODE \eqref{ode} is of first order, and we need to additionally specify a boundary condition in order to find a unique candidate solution.
To find an appropriate boundary condition, we use a discretization
argument as follows. 
If the control $U$ is restricted to take values in a discrete set 
$\{\frac{k}{N}, k=0,...,N\}$ (this is the case studied more thoroughly in Section~\ref{sec8} below), then the control problem reduces to a recursive optimal stopping problem. In particular, 
when $u=(N-1)/N$, there is only one remaining exercise right, and
the decision-maker faces a problem 
\begin{equation}\label{ost}
\sup_{\tau\geq 0}\E^u_\pi\left[e^{-r\tau}(\Pi^u_\tau -k)\frac{1}{N}\right],
\end{equation}
where the super-index $u$ in $\mathbb P_\pi^u$ indicates that 
$\Pi^u$ is the conditional probability as in \eqref{PiSDE} but with 
a constant signal-to-noise ratio $\rho(\frac{N-1}{N})$.
This problem corresponds to an irreversible investment problem with incomplete information {\em without} the learning-by-doing feature (i.e., in which the signal-to-noise ratio $\rho(u)$ is constant), and it is reasonable to expect that the limit (as $u\to 1$) of the optimal boundary in the stopping problem \eqref{ost} coincides with $b(1-)$.
%\st{Since there is no possibility to improve learning when $u=1$,  we consider the corresponding irreversible investment problem with incomplete information} {\em without} \st{the learning-by-doing feature (i.e., in which the signal-to-noise ratio} $u\mapsto \rho(u)$ \st{is constant). It is intuitively clear that the two problems should coincide in the limit when $u\to 1$.} 
We will thus use the obtained value of the boundary at $u=1$ for the problem with constant learning rates as the boundary condition for \eqref{ode}.

\subsection{Constant learning rates} \label{subsecconstant}

In this section, we study the simpler version of our problem where the signal-to-noise ratio is constant (i.e., investing more does not provide an improvement in the learning rate) and the decision-maker can only choose the time when to fully invest in the project. More precisely, for any fixed $u\in[0,1]$, consider the stopping problem 
\begin{equation}\label{vstopping}
 v(\pi)=v(\pi;u):=\sup_{\tau\geq 0}\E^u_\pi\left[e^{-r\tau}(\Pi^u_\tau-k)\right], \qquad \pi\in(0,1),
\end{equation}
where the super-indices $u$ denote that 
the signal-to-noise ratio $\rho(u)\in(0,\infty)$ is constant, and the supremum is taken over $\bF$-stopping times. 

\begin{remark}
We have the relation $V(u,\pi)\geq (1-u)v(\pi;u)$, 
and the gap in the inequality represents the additional value that learning-by-doing provides in the problem formulation \eqref{V} compared to a case with a constant learning rate. 
\end{remark}

By standard methods of optimal stopping, one finds a candidate value function by solving the following free-boundary problem: construct 
$(\hat v,c)$ such that 
\begin{equation}
\label{fbp2}
\left\{\begin{array}{rl}
\frac{\rho^2(u)}{2}\pi^2(1-\pi)^2 \hat v_{\pi\pi}-r\hat v =0, & \pi <c(u)\\
\hat v=\pi-k, & \pi=c(u)\\
\hat v_{\pi} =1, & \pi=c(u)\\
\hat v(0+)=0.\end{array}\right.
\end{equation}
The general solution of the ODE in \eqref{fbp2}, combined with the boundary condition at $\pi=0+$, is given by
\[\hat v(\pi) = D(u)G(u,\pi),\]
where $G$ is as in \eqref{G} above and $D$ is an arbitrary function.
The two boundary conditions at $\pi=c(u)$ then yield 
\[\begin{cases}
  D(u)G(u,c(u))=c(u)-k\\
D(u)G_\pi(u,c(u))=1,  
\end{cases}
\]
and using 
\[G_\pi(u,\pi)=\frac{\gamma(u)-\pi}{\pi(1-\pi)}G(u,\pi)\]
we find that
\begin{equation}\label{c}
c(u)=\frac{k\gamma(u)}{k+\gamma(u)-1}.
\end{equation}
The candidate value function is thus given by
\begin{equation}\label{hatv}
\hat v(\pi)=\left\{\begin{array}{cl}
    \frac{c(u)-k}{G(u,c(u))}G(u,\pi),  & \pi<c(u)\\
\pi-k, & \pi\geq c(u),
\end{array}\right.
\end{equation}
with $c$ as in \eqref{c}.

Since $\hat v$ is convex, we have $\hat v\geq \pi-k$. Moreover, $c\geq k$, which implies that  
\[\frac{\rho^2(u)}{2}\pi^2(1-\pi)^2 \hat v_{\pi\pi}-r\hat v \leq 0\]
for $\pi\not= c$.
Using standard methods from optimal stopping theory
(see, e.g., \cite{PS}) the verification of $\hat v=v$ is then straightforward
and we omit the proof.

\begin{proposition}
    Let $v$ be the value function defined as in \eqref{vstopping}, and let $\hat v$ be defined as in \eqref{hatv}. Then $v=\hat v$, and moreover, the stopping time $\tau_c:=\inf\{t\geq 0:\Pi^u_t\geq c(u)\}$ is optimal for \eqref{vstopping}.
\end{proposition}

\section{A study of the ODE for the boundary}
\label{sec5}

In this section we study the ODE \eqref{ode}.
In particular, we first show that, when paired with its boundary condition derived in the previous section, it has a unique 
solution. Moreover,
the heuristic derivation of \eqref{ode} uses the assumption that the boundary $b$ is monotone, %(cf. the boundary conditions along the boundary)
so we also provide conditions under which 
the solution of the ODE is indeed monotone. 

Consider the differential problem
\begin{equation}
    \label{ode+bc}
    \begin{cases}
        b'(u)= F(u,b(u)), \quad u\in(0,1)\\
        b(1)= c(1),
    \end{cases}
\end{equation}
where we recall from \eqref{F} that 
\begin{equation}
\label{F2}
F(u,b)=\frac{2(b-k)(\gamma')^2+\left(\gamma k-(\gamma+k-1)b\right)\gamma''}
{-\gamma(\gamma k-(\gamma+k-1)b)-(\gamma-1)(1-k)b}\cdot\frac{b(1-b)}{\gamma'}
\end{equation}
and where the boundary condition at $u=1$ is given by \[b(1)=c(1)=\frac{k\gamma(1)}{k+\gamma(1)-1}>k,\] cf.~\eqref{c}.

\begin{proposition}\label{Propbexist}
The ODE \eqref{ode+bc} has a unique solution $b$ on $[0,1]$. Moreover,
\begin{equation}\label{bounds}
    0< b(u)< c(u)
\end{equation}
for $u\in[0,1)$.
\end{proposition}

\begin{proof}
First note that the denominator of $F$ is bounded away from 0 on the region 
\[\mathcal O:=\{(u,b)\in[0,1]^2:b\leq c(u)\}.\]
Let $\tilde F$ be a modification of $F$ which coincides with $F$ on $\mathcal O$, and is Lipschitz
continuous on $[0,1]\times\R$.
It then follows from an application of the Picard-Lindel\"of theorem the existence of a unique solution $\tilde b$ of 
\begin{equation}
    \begin{cases}
        \tilde b'(u)= \tilde F(u,\tilde b(u)), \quad u\in(0,1)\\
        \tilde b(1)= c(1).
    \end{cases}
\end{equation}

By straightforward differentiation, 
\[c'(u)=\frac{-k(1-k)\gamma'}{(\gamma+k-1)^2}\]
and 
\[F(u,c(u))= \frac{2(c-k)\gamma'(1-c)}
{-(\gamma-1)(1-k)} =2c'(u) >c'(u).\]
Consequently, for every $u\in[0,1]$,
\begin{equation}
    \label{der}
    F(u,c(u))>c'(u)>0. 
\end{equation}
Therefore, 
$\tilde b(u)\leq c(u)$ for all $u\in[0,1]$. 
Indeed, assuming that
\[u_0:=\sup\{u\in[0,1):\tilde b(u)=c(u)\}\geq 0,\]
we must have, by continuity, 
\[F(u_0,c(u_0))=F(u_0,\tilde b(u_0))=\tilde b'(u_0)\leq c'(u_0).\]
However, this contradicts \eqref{der}, which proves that 
$\tilde b(u)< c(u)$ for all $u\in[0,1)$. 

Similarly, $F(u,b)\leq Db$ for some constant $D>0$, so by comparison we find that $\tilde b(u)\geq \tilde b(1)e^{-D(1-u)}>0$. Since $(u,\tilde b(u))\in \mathcal O$, and since $\tilde F\equiv F$ on $\mathcal O$, the result follows.
\end{proof}

We next study monotonicity properties of $b$.
To do so, we need to investigate the sign of the 
function $F$ in \eqref{F2}. 
Recall, from Proposition \ref{Propbexist}, that $0< b(u)\leq c(u)$ for every $u\in[0,1]$. As a consequence,
$$-\gamma(\gamma k-(\gamma+k-1)b(u))-(\gamma-1)(1-k)b(u)<0$$
and so the sign of $b'(u)=F(u,b(u))$ coincides with the sign of the function
\begin{align}\label{eq:H}
    H(u,\pi):= & \: 2(\pi-k)(\gamma')^2+\left(\gamma k-(\gamma+k-1)\pi\right)\gamma''\\
    =& \: \left(2(\gamma')^2-\gamma''\gamma+(1-k)\gamma'' \right)\pi-\left(2(\gamma')^2-\gamma''\gamma\right)k\nonumber
\end{align}
evaluated at $(u,b(u))$.
In particular, $b'(u)>0$ if and only if $H(u,b(u))>0$. 

Note that $H$ is affine in $\pi$, with 
$H(u,c(u))=2(c(u)-k)(\gamma')^2>0$ and 
$H(u,0)=-(2(\gamma')^2-\gamma''\gamma)k$. 
Consequently, if $H(u,0)\geq 0$, then $b$ is automatically monotone increasing.

\begin{proposition}
Assume that, for every $u\in[0,1]$,
\begin{equation}
    \label{cond1}
2(\gamma'(u))^2-\gamma(u)\gamma''(u)\leq 0.
\end{equation}
Then, the solution $b$ of \eqref{ode+bc} satisfies $b'(u)>0$
for all $u\in[0,1]$.
\end{proposition}

\begin{proof}
If \eqref{cond1} holds, then $F(u,\pi)>0$ at all points
$(u,\pi)$ with $0<\pi<c(u)$. Consequently, $b'(u)>0$.
\end{proof}

\begin{remark}
    We note that \eqref{gamma} can be used to translate
    the condition \eqref{cond1} into a condition directly on the signal-to-noise ratio $\rho(u)$.
    %; this condition is rather involved, however, and we choose not to include it.
    % In fact, \eqref{cond1} holds precisely if 
    % \[\frac{r}{\rho^2(u)}\left(\left(\frac{1}{\rho^2(u)}\right)'\right)^2-\left(\frac{1}{2}+\frac{2r}{\rho^2(u)}+\sqrt{\frac{1}{4}+\frac{2r}{\rho^2(u)}}\right)\left(\frac{1}{8r}+\frac{1}{\rho^2(u)}\right)\left(\frac{1}{\rho^2(u)}\right)''\leq 0\]
    % for all $u\in[0,1]$. 
\end{remark}

Next, assume that $2(\gamma'(u))^2-\gamma(u)\gamma''(u)> 0$ for every $u\in[0,1]$ and define
\begin{equation}\label{B}
B(u):=\frac{2(\gamma'(u))^2-\gamma(u)\gamma''(u)}{2(\gamma'(u))^2-\gamma(u)\gamma''(u)+(1-k)\gamma''(u) }k, \qquad u\in[0,1].
\end{equation}
Then, $0<B(u)<c(u)$ and $H(u,B(u))=0$. 

\begin{proposition}\label{Propbmonot}
Assume that, for every $u\in [0,1]$,
\begin{equation}\label{cond2}
2(\gamma'(u))^2-\gamma(u)\gamma''(u)> 0,
 %\qquad u\in[0,1]
\end{equation}
and $B'(u)>0$. Then, the solution $b$ of \eqref{ode+bc} satisfies $b'(u)>0$ and $b(u)>B(u)$ for all $u\in[0,1]$.
\end{proposition}

\begin{proof}
Under the condition \eqref{cond2} we have 
$H(u,0)<0$, and consequently
$F(u,\pi)>0$ if and only if $\pi>B(u)$.
Therefore, it suffices to show that $b(u)>B(u)$. 

Define 
    \[u_0:=\sup\{u\in[0,1]:b(u)=B(u)\}\]
    and assume, to reach a contradiction, that $u_0\geq 0$. Since $b(1)=c(1)>B(1)$, by continuity, 
    we must have $u_0<1$. Moreover, by the definition 
    of $u_0$, we must have $b'(u_0)\geq B'(u_0)$. However, this contradicts 
    \[b'(u_0)=F(u_0,b(u_0))=0<B'(u_0).\] 
    Consequently, $b(u)> B(u)$ for all $u\in[0,1]$ and so $b'(u)>0$.
\end{proof}

%\begin{remark}
%    A closer inspection of the proof above reveals that 
%   if $B'>0$, then also $b'(u)>0$ for all $u\in[0,1]$.
%\end{remark}

\begin{remark}
     We note that condition \eqref{cond1} requires that $\gamma$ is ``sufficiently'' convex, whereas condition \eqref{cond2} holds when $\gamma$ is either concave or ``mildly'' convex. This depends, from \eqref{eq:gamma}, on the form of the signal-to-noise ratio.
\end{remark}

Since it will be important to determinate whether $b(u)\geq k$ for every $u\in[0,1]$ (see Proposition \ref{ProphatV} and Theorem~\ref{main} below), we enunciate the following corollary.

\begin{corollary}\label{cor:gammaconc}    
Assume that $\gamma$ is concave and $B'>0$. Then, the solution $b$ of \eqref{ode+bc} satisfies $b'(u)>0$ and $b(u)>k$ for all $u\in[0,1]$.
\end{corollary}

\begin{proof}
    The result directly follows from Proposition \ref{Propbmonot} and the fact that, if $\gamma$ is concave, then $B(u)>k$ for every $u\in[0,1]$.
\end{proof}

To guarantee the monotonicity of $B$ needed for Proposition \ref{Propbmonot} and Corollary \ref{cor:gammaconc}, we have the following simple result.

\begin{proposition}\label{prop:Bincreas}
    Assume that $\gamma$ is $C^3([0,1])$, that $\gamma''(u)<0$ for $u\in[0,1]$, and that $\gamma$ satisfies 
    \begin{equation}\label{cond gamma}
   3 (\gamma''(u))^2< 2\gamma'(u)\gamma'''(u)
    \end{equation}
    for all $u\in[0,1]$. 
    Then, $B$ in \eqref{B} is strictly increasing, i.e., $B'>0$.
\end{proposition}

\begin{proof}
From $H(u,B(u))=0$, we have
$$B'(u)=-\frac{H_u(u,B(u))}{H_{\pi}(u,B(u))}.$$
If $\gamma$ is concave, then $H_{\pi}(u,\pi)>0$ so it suffices to show that 
$H_u(u,B(u))<0$.
Differentiation yields 
\begin{align}
    H_u(u,B(u)) &= 3(B(u)-k)\gamma'\gamma'' +(\gamma k-(\gamma+k-1)B(u))\gamma'''\\
    &= 3(B(u)-k)\gamma'\gamma'' -2(B(u)-k)\frac{(\gamma')^2\gamma'''}{\gamma''}\\
    &= (B(u)-k)\frac{\gamma'}{\gamma''}(3(\gamma'')^2-2\gamma'\gamma''')<0,
\end{align}
where we used in the second equality that $H(u,B(u))=0$, and the inequality follows from \eqref{cond gamma} and the fact that $B\geq k$ since $\gamma$ is concave. It follows that $B'>0$.
\end{proof}

\begin{remark}
     Figure \ref{F:nonmonotone b} shows that the solution $b$ to the ODE \eqref{ode+bc} is not always monotone.
\end{remark}

\begin{figure}[ht]
\centering
\includegraphics[scale=0.6]{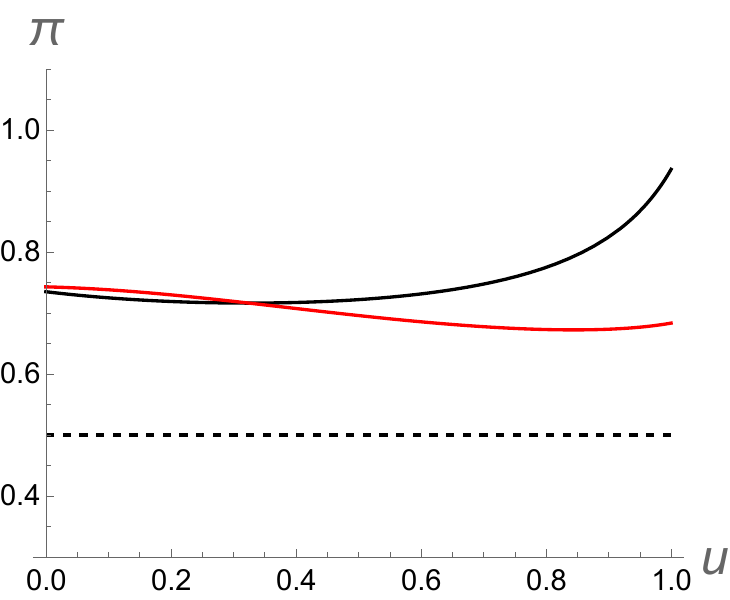}
\caption{The solution $b$ to the ODE \eqref{ode+bc} (solid black), the curve $B$ (red) and the threshold $k$ (dashed black), in the case $\rho^2(u)=\frac{1}{4 (1-0.1u-0.8u^2)}$, $k=0.5$ and $r = 0.1$.} 
\label{F:nonmonotone b}
\end{figure}

\section{Verification}
\label{sec6}
We now formally construct the candidate optimal strategy, heuristically introduced in \eqref{heur}, that performs reflection along the boundary $b$ (recall Figure \ref{F:simulation}). 
To do that, assume that the solution $b$ of  \eqref{ode+bc} is strictly increasing; sufficient conditions for this monotonicity were provided in Section~\ref{sec5} above. Recall that $h=b^{-1}$ denotes the inverse of $b$; it is defined on $[b(0),b(1)]$,
and we extend it to $(0,1)$ so that $h(\pi)=1$ for $\pi>b(1)$ and 
$h(\pi)=0$ for $\pi<b(0)$.

For any fixed $(u,\pi)\in[0,1]\times(0,1)$, we define the candidate optimal strategy $\hat U$, to perform reflection along $b$, as follows. Denote by $C([0,\infty))$ the space of continuous functions from $[0,\infty)$ to $[0,1]$ and define the map $\tilde U:[0,\infty)\times C([0,\infty))\to[0,1]$ by
\[\tilde U_t(\omega):= u\vee h\bigg(\sup_{0\leq s\leq t}\omega_s\bigg),\]
which will serve as the feed-back map of the optimal control. Now consider (cf.\ \eqref{PiSDE}) the stochastic differential equation (SDE)
\begin{equation}
    \label{P}
    \ud P_t=-\rho^2(\tilde U_t(P))P^2_t(1-P_t)\,\ud t + 
    \rho(\tilde U_t(P))P_t(1-P_t)\,\ud X_t,
\end{equation}
with $P_0=\pi$. The drift and diffusion coefficients of the SDE \eqref{P} satisfy the (locally) Lipschitz conditions of, e.g., \cite[Ch.~V, Th.~12.1]{rogers2000diffusions} and thus the SDE \eqref{P} admits a unique strong solution $P=(P_t)_{t\geq 0}$. Then, define the candidate optimal control by
\begin{equation}
    \label{Uhat}
   \hat{U}_{0-}=u \quad \text{and} \quad \hat U_t:= \tilde U_t(P), \quad t\geq 0.
\end{equation}
Since $P$ is $\bF$-adapted, we have that $\hat U\in\A_u$, as defined in \eqref{A_u}.
Recall that, by \eqref{PiSDE}, we also have 
\[\ud \Pi^{\hat U}_t=-\rho^2(\hat U_t)(\Pi^{\hat U}_t)^2(1-\Pi^{\hat U}_t)\,\ud t + 
    \rho(\hat U_t)\Pi^{\hat U}_t(1-\Pi^{\hat U}_t)\,\ud X_t
\]
i.e., $\Pi^{\hat U}$ satisfies \eqref{P} and so, by uniqueness, $\Pi^{\hat U}$ and $P$ are indistinguishable. Notice that this is in line with our conjecture in \eqref{heur}.

Next, we define the candidate value function $\hat V:[0,1]\times(0,1)\to\R$ by
\begin{equation}
    \label{hatV}
    \hat V(u,\pi):=\begin{cases}
        A(u)G(u,\pi), & \pi\leq b(u)\\
    A(h(\pi))G(h(\pi),\pi)+ (\pi-k)(h(\pi)-u), & \pi>b(u),
    \end{cases}
\end{equation}
where 
\[A(u)= \frac{(\gamma(u)+k-1)b(u)-\gamma(u) k}{\gamma'(u)G(u,b(u))},\]
(cf.~\eqref{A}) and $G$ is as in \eqref{G}.
In this way, $\hat V$ is continuous. We now show some further properties it satisfies, which are essential to obtain the Verification theorem.

\begin{proposition}\label{ProphatV}
Assume that $b$ is strictly increasing on $[0,1]$. 
We have that
    $$\hat V\in C^{1,2}([0,1]\times(0,1)),$$ %\quad \text{and} \quad \hat V_{\pi\pi}\in C([0,1]\times(0,1)),$$
    with $\hat V_u\leq k-\pi$. 
  Moreover, if $b(u)\geq k$ for every $u\in[0,1]$ or if \eqref{cond1} holds, then
    \begin{equation}\label{LV}
    \frac{\rho^2}{2}\pi^2(1-\pi)^2 \hat V_{\pi\pi}-r\hat V \leq 0.
    \end{equation}
\end{proposition}

\begin{proof}
First, we study differentiability of $\hat V$. It is clear that $\hat V$ is of class $C^1$ below the boundary. Moreover,  
$\hat V_u(u,b(u)-)=k-b(u)$ by construction (recall \eqref{fbp}), and since $\hat V$ is extended linearly in $u$ with slope $k-\pi$ for $\pi>b(u)$, it follows that $\hat V_u$ is continuous.
More precisely, for $(u,\pi)$ with $b(u)<\pi$, we have
$\hat V(u,\pi)=\hat V(u_0,\pi)+(\pi-k)(u_0-u)$ where $u_0:=h(\pi)$, and so
\begin{equation}\label{piderivative}
\hat V_\pi(u,\pi)=\hat V_{\pi}(u_0,\pi)+u_0-u
\end{equation}
since $\hat V_u(u_0,\pi)=k-\pi$. Thus, $\hat V\in C^1([0,1]\times(0,1))$.

We now check that $\hat V_u\leq k-\pi$. We clearly have $\hat V_u= k-\pi$ above the boundary, so it remains to treat points below the boundary. 
For $\pi<b(u)$, we have
$$\hat V_u(u,\pi)=A'(u) G(u,\pi)+A(u) G_u(u,\pi)$$
and
$$\hat{V}_{u\pi}(u,\pi)=A'(u)G_{\pi}(u,\pi)+A(u)G_{u\pi}(u,\pi).$$
Since
$$G_{\pi}(u,\pi)=\frac{\gamma(u)-\pi}{\pi(1-\pi)}G(u,\pi)\quad \text{and} \quad G_{u\pi}(u,\pi)=\frac{\gamma(u)-\pi}{\pi(1-\pi)}G_u(u,\pi)+\frac{\gamma'(u)}{\pi(1-\pi)}G(u,\pi),$$
we obtain that
\begin{equation}\label{eq:V_upi}
    \hat{V}_{u\pi}(u,\pi)=\frac{\gamma(u)-\pi}{\pi(1-\pi)}\hat{V}_u(u,\pi)+\frac{\gamma'(u)}{\pi(1-\pi)}A(u)G(u,\pi).
\end{equation}

Now assume, to reach a contradiction, that there exists $(u,\pi_0)\in[0,1]\times(0,1)$ with $\pi_0<b(u)$ such that $\hat{V}_u(u,\pi_0)>k-\pi_0$. We then obtain from \eqref{eq:V_upi} and by \eqref{A} that
\begin{align*}
    \hat{V}_{u\pi}(u,\pi_0)&>\frac{\gamma(u)-\pi_0}{\pi_0(1-\pi_0)}(k-\pi_0)+\frac{(\gamma(u)+k-1)b(u)-k\gamma(u)}{\pi_0(1-\pi_0)}\frac{G(u,\pi_0)}{G(u,b(u))}\\
    &\geq \frac{\gamma(u)-\pi_0}{\pi_0(1-\pi_0)}(k-\pi_0)+\frac{(\gamma(u)+k-1)b(u)-k\gamma(u)}{\pi_0(1-\pi_0)}\\
    &=\frac{(\gamma(u)+k-1)(b(u)-\pi_0)}{\pi_0(1-\pi_0)}-1>-1.
\end{align*}
Consequently, $\pi\mapsto\hat V_{u}(u,\pi)+\pi-k$ is positive and increasing 
on $(\pi_0,b(u))$, which contradicts $\hat V_u(u,b(u))=k-b(u)$. It follows that $\hat{V}_u\leq k-\pi$ everywhere.

Differentiating \eqref{piderivative} once more with respect to $\pi$ and using $\hat V_{u\pi}=-1$ along the boundary (recall \eqref{fbp}), yields
\[\hat V_{\pi\pi}(u,\pi)=\hat V_{\pi\pi}(u_0,\pi),\]
which shows that $\hat V_{\pi\pi}$ is continuous. 

Finally, \eqref{LV} holds with equality below the boundary by construction, and above the boundary we have 
\begin{align*}
\frac{\rho^2(u)}{2}&\pi^2(1-\pi)^2 \hat V_{\pi\pi}(u,\pi)-r\hat V(u,\pi) \\
&= \frac{\rho^2(u)}{2}\pi^2(1-\pi)^2 \hat V_{\pi\pi}(u_0,\pi)-r(\hat V(u_0,\pi) +(\pi-k)(u_0-u))
\\
&= \left(\frac{\rho^2(u)}{\rho^2(u_0)}-1\right)rA(u_0)G(u_0,\pi)-r(\pi-k)(u_0-u).
\end{align*}
Thus, if $\pi\geq b(u)\geq k$, then both terms are negative, and 
\eqref{LV} follows.

Similarly, if \eqref{cond1} holds, using the expression 
\eqref{A} for $A$, we need to check that 
\begin{align*}
&\left(\frac{\rho^2(u)}{\rho^2(u_0)}-1\right)\frac{\gamma_0 k-(\gamma_0+k-1)\pi}{-\gamma'_0}+(\pi-k)(u-u_0)\\
&= \left(\frac{\gamma^2_0-\gamma_0}{\gamma^2(u)-\gamma(u)}-1\right)\frac{\gamma_0 k-(\gamma_0+k-1)\pi}{-\gamma'_0}+(\pi-k)(u-u_0)
\leq 0,
\end{align*}
where $\gamma_0:=\gamma(u_0)$ and 
$\gamma'_0:=\gamma'(u_0)$. Since $\gamma>\gamma_0$, we have
\[\frac{\gamma^2_0-\gamma_0}{\gamma^2-\gamma}\leq
\frac{\gamma_0}{\gamma},\]
so it then suffices to show that
\[f(u):=\left(\frac{\gamma_0}{\gamma(u)}-1\right)\frac{\gamma_0 k-(\gamma_0+k-1)\pi}{-\gamma'_0}+(\pi-k)(u-u_0)\leq 0.\]
However, it is clear that $f(u_0)=0$; also, 
\[f'(u_0)=\frac{(1-k)\pi}{\gamma_0}>0\]
and $f$ is concave by \eqref{cond1} and the fact that $\pi=b(u_0)\leq c(u_0)$.
Consequently, 
$f(u)\leq 0$ for $u\leq u_0$, and \eqref{LV} holds.
\end{proof}

The results of Proposition \ref{ProphatV} lead to the Verification theorem, which we now present.

\begin{theorem}\label{main}
Let $b$ be the solution of the differential problem \eqref{ode+bc}. Assume that $b$ is strictly increasing and either $b(u)\geq k$ for every $u\in [0,1]$ or \eqref{cond1} holds. Then, 
$V=\hat V$ and the strategy $\hat U$ is optimal in \eqref{V}.
\end{theorem}

\begin{proof}
Let $U\in\mathcal{A}_u$ be an arbitrary strategy and let
\[
Y_t:=Y^U_t:=e^{-rt}\hat V (U_t,\Pi_t^U) + \int_0^t e^{-rs}\Big(\Pi_s^U-k\Big) \,\ud U_s 
\] 
for $t\geq 0-$.
By Proposition \ref{ProphatV}, we can apply Itô's formula for semimartingales with jumps (see, e.g., \cite[Th.\ 3.1]{peskirchangeoftime}) to $Y$ and obtain 
\begin{align}\label{eq:ItoY}
\ud Y_t 
 &=  e^{-rt}\bigg( \frac{1}{2}\rho^2(U_{t-})\Pi^2_{t}(1-\Pi_t)^2\hat V_{\pi\pi}(U_{t-},\Pi_t)-r\hat V(U_{t-},\Pi_t) \bigg)\,\ud t+e^{-rt}(\Pi_t-k)\ud U_t \nonumber\\
&\hspace{12pt}+ e^{-rt} \hat V_u (U_{t-},\Pi_t)\,\ud U^c_t 
  + e^{-rt}\left(\hat V(U_t,\Pi_t) - \hat V(U_{t-},\Pi_{t})\right)+ e^{-rt}\hat V_\pi(U_{t-},\Pi_t) \,\ud\Pi_t,
\end{align}
where $U^c$ denotes the continuous part of $U$ and $\Pi:=\Pi^U$. 
Here, the first term is non-positive by Proposition~\ref{ProphatV}, and 
$\hat V_u\leq k-\pi$ gives that the next three ones are non-positive together.
Moreover, (recall \eqref{PiSDE}) 
$$\int_0^te^{-rs}\hat V_\pi(U_{s-},\Pi_s) \ud\Pi_s
=\int_0^te^{-rs}\rho(U_s)\Pi_s(1-\Pi_s)\hat V_\pi(U_{s-},\Pi_s)\ud \hat{W}^U_s$$ 
is a $(\bP^U,\bF)$-martingale since $\hat V_\pi$ and $\rho$ are bounded. Thus, the process $Y$ is a $(\bP^U,\bF)$-supermartingale on $[0-,\infty)$, and since $Y$ is lower bounded it is also a $(\bP^U,\bF)$-supermartingale on $[0-,\infty]$. It follows that
\begin{eqnarray*}
\hat V (u,\pi)  = Y_{0-} \geq \E_\pi^U\big[Y_\infty\big] =  \E_\pi^U\bigg[ \int_0^\infty e^{-rt}\big(\Pi_t^U-k\big) \,\ud U_t \bigg] .
\end{eqnarray*}
Since $U\in\mathcal A_u$ is arbitrary, it follows that $\hat V \geq V$.

To prove the reverse inequality, we consider the strategy $\hat U$ as in \eqref{Uhat} (recall also \eqref{heur} for an explicit form), which is continuous except for a potential jump at $0$. Denote $\Pi:=\Pi^{\hat U}$, so that $(\hat U,\Pi)$ always stays below the boundary $b$ at all times $t$ with $0\leq t\leq\inf\{s\geq 0:\hat U_s=1\}$. Then, by construction,
\[
\frac{1}{2}\rho^2(\hat U_t)\Pi_t^2(1-\Pi_t)^2\hat V_{\pi\pi}(\hat U_t,\Pi_t) -r\hat V(\hat U_t,\Pi_t)  = 0
\]
and 
\[
\hat V_u(\hat{U}_t,\Pi_t)\,\ud \hat U^c_t  = (k-\Pi_t)\,\ud \hat U^c_t.
\]
Moreover, at $t=0$, if the initial point $(u,\pi)$ satisfies $u<h(\pi)$, there occurs an initial and bounded jump in $\hat U$ of size $\ud \hat U_0 = h(\pi)-u$, but no additional jumps occur. Since $\hat V(u,\pi)=\hat V(h(\pi),\pi)+(\pi-k)(h(\pi)-u)$ for $u<h(\pi)$, we have that 
\[e^{-rt}(\Pi_t-k)\ud \hat U_t 
+ e^{-rt} \hat V_u (\hat U_{t},\Pi_t)\,\ud \hat U^c_t 
  + e^{-rt}\left(\hat V(\hat U_t,\Pi_t) - \hat V(\hat U_{t-},\Pi_{t})\right)=0.\]
Thus, by Itô's formula, the process $Y=Y^{\hat U}$ is a $(\bP^{\hat U},\bF)$-martingale. Since it is bounded, it is a $(\bP^{\hat U},\bF)$-martingale also on $[0-,\infty]$. It follows that
\begin{eqnarray*}
\hat V (u,\pi) = Y_0^{\hat U} = \E_\pi^{\hat{U}}\big[Y_\infty\big] = \E_\pi^{\hat{U}}\bigg[\int_0^\infty e^{-rt}\left(\Pi_t-k\right)\,\ud \hat U_t \bigg].
\end{eqnarray*}
Consequently, $\hat V \leq  V$.

Combining the two inequalities, it follows that $V\equiv \hat V$, and $\hat U$ is optimal in \eqref{V}.
\end{proof}

\begin{remark}
    Theorem \ref{main} shows that we can determine the solution to our problem when $b$ is increasing and either $b(u)\geq k$ for every $u\in[0,1]$ or \eqref{cond1} holds. Whether these conditions are satisfied depend on the form of the signal-to-noise ratio $\rho$ and in Section~\ref{sec5} we have obtained some sufficient conditions that satisfy the hypotheses of Theorem~\ref{main} (recall, e.g., Corollary \ref{cor:gammaconc} and Proposition~\ref{prop:Bincreas}). In the next section we will provide some specific forms of $\rho$ that fulfill the aforementioned conditions and, in particular, we will show that there are some choices of $\rho$ under which \eqref{cond1} holds but $b(u)< k$ for every $u\in[0,u_0)$ and some $u_0\in(0,1)$ (see Figure \ref{fig:b<k}).
\end{remark}

\section{Examples}\label{sec7}

In this section we provide a few 
examples for our model.

\begin{example}\label{subsection:cashflows} \textbf{(Relation to a problem with observable noisy cash flows).} Assume that a project 
pays during a time interval $[t,t+dt)$ an investor with investment level $U_t$ the amount $U_t \ud Y_t$, where 
\[\ud Y_t=\mu \,dt +\sigma\,\ud B_t\]
represents the profit per total size of the project. Here $\mu$ is a random variable with a two-point distribution on $\{\mu_0,\mu_1\}$ that is independent of the Brownian motion $B$, and $\sigma$ is a positive constant. 

To introduce the learning-by-doing feature, the investor with investment level $U_t$ is able to decompose 
the Brownian motion $B$ as
\[\ud B_t=f(U_t) \,\ud W_t + \sqrt{1-f^2(U_t)}\,\ud W^{obs}_t,\]
where $W$ and $W^{obs}$ are independent Brownian motions, with $W^{obs}$ being observable for the decision-maker, and $f:[0,1]\to[0,1)$ is a decreasing function with $f(0)=1$. Then, the signal-to-noise ratio becomes 
\[\rho(U_t)=\frac{\mu_1-\mu_0}{\sigma f(U_t)}.\]
The total value associated with an investment strategy $U$, given $U_0=0$, is then
\begin{eqnarray*}
\E\left[\int_0^\infty e^{-rt}U_t\,\ud Y_t\right]
= \E\left[\int_0^\infty e^{-rt}\mu U_t\,\ud t\right] = \frac{1}{r}\E\left[\int_0^\infty e^{-rt}\mu \,\ud U_t\right],
\end{eqnarray*}
where the second equality uses integration by parts.
We note that the last expression coincides (up to a multiplicative factor $r^{-1}$) with the formulation in \eqref{formulation}, where in this case 
\[\ud X_t:=\frac{1}{\sigma f(U_t)}\left(\ud Y_t-\mu_0\ud t-\sigma\sqrt{1-f^2(U_t)}\ud W^{obs}\right)\]
is in the form \eqref{eq:startX} with $\rho(U_t)=\frac{\mu_1-\mu_0}{\sigma f(U_t)}$. 
\end{example}

\begin{example}\label{ex:ProjExp} {\bf (Project expansion).}
In this example we discuss a simplistic model for project expansion. 
To do that, assume that a decision-maker 
runs a business with unknown value $\mu$ and has access to noisy observations 
described by
\[\ud X^0_t=\rho\theta\,\ud t+\ud W^0_t,\]
where $\rho> 0$ is a given constant, 
$\theta = (\mu-\mu_0)/(\mu_1-\mu_0)$ and $W^0$ is a Brownian motion.
Moreover, assume that the decision-maker has the possibility to expand his/her activities by starting another identical business, but with independent noise. Thus, in addition to $\ud X^0_t$, observations of
\[\ud X^1_t =\rho\theta\,\ud t+\ud W^1_t\]
become available after expansion, 
where $W^1$ is a Brownian motion independent of $W^0$.
Note that the drifts contain the same random factor $\theta$, and thus the learning rate is larger after expansion. More specifically, observing $\ud X^0_t$ and $\ud X^1_t$ provides the same information as observing $\ud X_t:=\ud X^0_t+\ud X^1_t=2\rho\theta\ud t+ \sqrt 2\ud W_t$, where $W$ is a Brownian motion.
Consequently, the signal-to-noise ratio 
increased from $\rho$ to $\sqrt 2\rho$ after expansion.

In a continuous setting, the above example generalizes to a signal-to-noise ratio $\rho(u)=C\sqrt{u}$. It is straightforward to check that \eqref{cond1} holds for this $\rho$, so Theorem~\ref{main} applies. 
\end{example}

\begin{example}{\bf (Inversely linear signal-to-noise ratio}).
For an illustration of the optimal reflecting boundary, we consider a special case in which 
%\st{$f^2(u)$ is linear, i.e., $f(u)=\sqrt{D_1-D_2u}$ with $0<D_2<D_1$. Then,$\rho^2(u)=\frac{C}{1-Du}$ for some constants $C>0$ and $D\in(0,1)$.}
the squared signal-to-noise ratio $\rho^2(u)$ is inversely linear, i.e.,
\[
\rho(u)=\sqrt{\frac{D_1}{1-D_2u}}
\]
for some constants $D_1>0$ and $D_2\in(0,1)$. 
%(This implies that $f^2(u)$ introduced in Example~\ref{subsection:cashflows} would be linear.) 
In this case, differentiation of \eqref{gamma} yields 
\begin{equation}
   \gamma''(u)=\frac{-2(\gamma'(u))^2}{2\gamma(u)-1}<0. 
\end{equation}
    Therefore, the ODE \eqref{ode+bc} becomes 
\begin{equation}
    \left\{\begin{array}{ll}
    b' = \frac{(3\gamma+k-2)b-(3\gamma-1)k}
{(\gamma+k-1)(\gamma-1)b-\gamma k(\gamma-b)}\cdot\frac{2\gamma'b(1-b)}{2\gamma-1}, & \quad u\in(0,1)\\
b(1)= c(1).\end{array}\right.
\end{equation}

We then have that the 0-level curve $B$ of $F$ is given by
\[B(u)=\frac{(3\gamma-1)k}{3\gamma+k-2}>k\]
and we note that $B$ is monotone increasing with $B'>0$.
The existence of a monotone solution $b$ to \eqref{ode+bc} thus follows by Corollary \ref{cor:gammaconc}, and 
we have that  $b(u)>k$ for every $u\in[0,1]$. Therefore, Theorem \ref{main} applies and the optimality of the  solution is verified. See Figure~\ref{F:linear case} for a plot of the optimal boundary $b$.
\end{example}

\begin{figure}[h]
\centering
\includegraphics[scale=0.6]{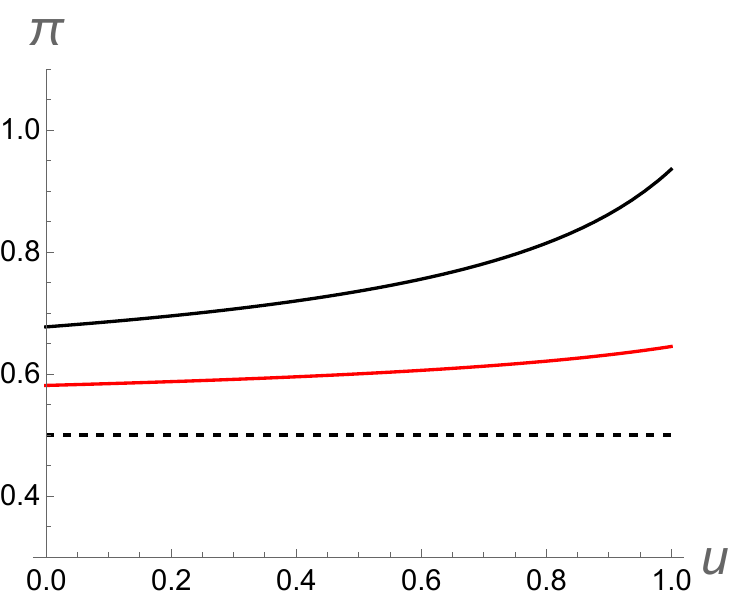}
\caption{The optimal reflecting boundary $b$ (solid black), the curve $B$ (red) and the threshold $k$ (dashed black) in the case 
$\rho^2(u)=\frac{1}{4(1-0.9u)}$, $k=0.5$ and $r=0.1$.
% $\rho(u)=(\mu_1-\mu_0)/(\sigma \sqrt{D_1-D_2u})$. Parameters: $\mu_0 = -0.25$, $\mu_1 = 0.25$, $r = 0.1$, $\sigma=1$,
% $D_1 = 1$, $D_2 = 0.9$. 
} 
\label{F:linear case}
\end{figure}

We next provide an example in which the optimal investment boundary $b$ goes below the level $k$. Notice, from \eqref{V}, that increasing the level of investment when $\Pi$ is below $k$ yields an instantaneous negative reward. The decision-maker should thus sometimes expand the project even though the current estimate of the project value is negative. 

\begin{example}{\bf ($\bm{b(u)<k}$).}
Consider the case
$$\gamma(u) = \frac{1.25}{u + 0.2}$$
for $u\in[0,1]$. It can be verified that $\gamma$ satisfies \eqref{cond1}. 
Hence $b'(u)>0$ for all $u\in[0,1]$, and Theorem \ref{main} applies. Figure \ref{fig:b<k} shows that the optimal boundary $b$ goes below the threshold $k$ for small values of $u$. 
This is remarkable since it corresponds to an instantaneous negative reward for such values (recall \eqref{V}). The explanation for this seemingly irrational behaviour is that the negative reward is compensated by a comparatively large value of future learning due to an increased learning rate.

% We interpret this behaviour as the decision-maker attempt to increase a signal-to-noise ratio that is almost null and thus does not allow for a significant learning of the true value of the project. Indeed, notice that in this example the corresponding signal-to-noise ratio is given by
% $$\rho^2(u)=\frac{0.2(u+0.2)^2}{(1.25)^2-1.25(u+0.2)},$$
% which is almost null for small $u$.
% }
\end{example}

\begin{figure}[ht]
\centering
\includegraphics[scale=0.6]{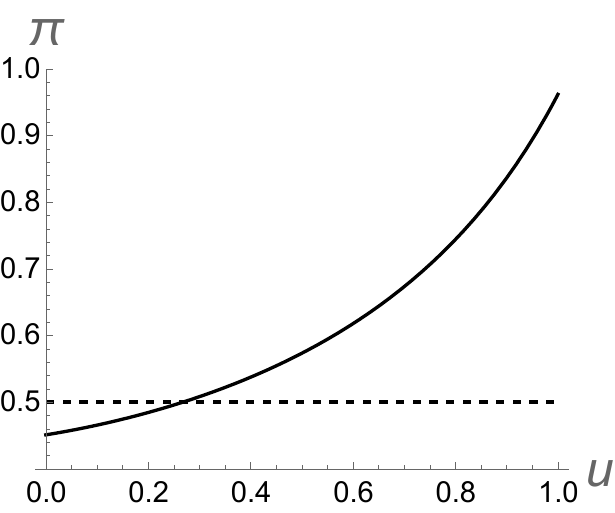}
\caption{The optimal reflecting boundary $b$ (solid black) and the threshold $k$ (dashed black) in the case $\gamma(u) = 1.25/(u + 0.2)$ and $k=0.5$. } 
\label{fig:b<k}
\end{figure}

\section{The discrete case}\label{sec8}
   
In this section we study a similar problem of irreversible investment under incomplete information and learning-by-doing, but where the control $U$ is restricted to take values in a discrete subset of $[0,1]$. In this setting, we analyze under what conditions the related investment boundary is monotone  in the number of remaining exercise rights, and we 
characterize the optimal strategy.
The analysis of the discrete case mirrors, and complements, the 
continuous version of the problem of irreversible investment, as presented above.

To introduce the problem, let an integer $N\geq 0$ be given, and define $u_n=\frac{n}{N}$ for $n=0,...,N$. We study the following recursively defined problem: 
\begin{equation}\label{discreteVn-1}
    \begin{cases}
        &V_N(\pi) = \sup_{\tau} \E_\pi ^{u_N}\Big[e^{-r\tau}\left(\Pi^{u_N}_\tau-k \right) \Big],\\
        & V_{n}(\pi) = \sup_{\tau} \E _\pi^{u_n}\Big[e^{-r\tau}\left(V_{n+1} (\Pi^{u_n}_\tau)+\Pi^{u_n}_\tau-k \right)\Big], \qquad n=0,...,N-1.
    \end{cases}
\end{equation}

\begin{remark}
The recursively defined optimization problem \eqref{discreteVn-1} is a discrete version of the continuous formulation \eqref{V}.
Indeed, if the set of admissible controls $\A$ is further restricted to take values only in $\{u_n\}_{n=0}^N$, then problem \eqref{V} reduces to 
a multiple stopping problem. 
% It can be shown that 
% \[
% \bE^U\bigg[\int_0^\infty e^{-rt}\mu U_t \ud t\bigg] = \frac{\mu_1-\mu_0}{rn}\bE^U\bigg[\sum_{i=1}^n e^{-r\tau_i}\left(\Pi^{U_i}_{\tau_i}-k\right)\bigg].
% \] 
By standard literature on Markovian multiple stopping problems (see, e.g., \cite{CD}), such problems can be formulated recursively as in \eqref{discreteVn-1}.
\end{remark}

As in the continuous case treated above, we first construct candidate solutions $\hat V_n(\pi)$, $n=0,...,N$, and we then verify that $\hat V_n=V_n$. The candidate solution is constructed using an Ansatz that there exists an increasing sequence $\{b_n\}_{n=0}^N$ such that 
\[\tau_{n}:=\inf\{t\geq 0:\Pi^{u_n}_t\geq b_n\}\]
is optimal for $V_n$. The candidate solutions will be described using the notation 
\[G_n(\pi):=G(u_n,\pi),\]
where $G$ is as in \eqref{G} with 
$\gamma=\gamma_n:=\gamma(u_n)$.
We also let 
\[c_n:=c(u_n)=\frac{\gamma_n k}{\gamma_n+k-1}.\]

\subsection{Solving the discrete problem}\label{sec8.1}

First consider the last step, i.e., the stopping problem 
\[V_N(\pi) = \sup_{\tau} \E_\pi ^{u_N}\Big[e^{-r\tau}\left(\Pi^{u_N}_\tau-k \right) \Big].\]
This is a payoff of call option type
on the process $\Pi^{u_N}$, and was already treated in Subsection~\ref{subsecconstant}.
In fact, 
\[ V_N(\pi)=\left\{\begin{array}{cl}
A_NG_N(\pi) & \pi<b_N\\
\pi-k & \pi\geq b_N,\end{array}\right.\]
where 
\[b_N=\frac{\gamma_N k}{\gamma_N+k-1}=:c_N\]
and $A_N$ is a constant.

% Now consider the second last step, i.e. the problem
% \[V_{N-1}(\pi) = \sup_{\tau} \E_\pi ^{u_{N-1}}\bigg[e^{-r\tau}\left(V_N(\Pi^{u_{N-1}}_\tau)+\Pi^{u_{N-1}}_\tau-k \right) \bigg].\]
% Again, a smooth fit guess yields a candidate value
% \[\hat V_{N-1}(\pi)=\left\{\begin{array}{cl}
% A_{N-1}G_{N-1}(\pi) & \pi<b_{N-1}\\
% V_N(\pi)+\pi-k & \pi\geq b_{N-1},\end{array}\right.\]
% where 
% \begin{equation}\label{syst}
%         \left\{\begin{array}{l}
%     A_{N-1}G_{N-1}( b_{N-1}) = b_{N-1}-k +V_{N}(b_{N-1}) \\
%     A_{N-1}G'_{N-1}(b_{N-1}) =1+ V'_{N}(b_{N-1}).
%     \end{array}\right.
% \end{equation}
% If $b_{N-1}\leq b_N$, then 
% \[V'_{N}(b_{N-1})=\frac{\gamma_{N}-b_{N-1}}{b_{N-1}(1-b_{N-1})}V_{N}(b_{N-1}),\]
% and \eqref{syst} leads to the equation 
% \begin{equation}
% %\label{bN-1}
% (\gamma_{N-1}+k-1)b_{N-1}-\gamma_{N-1}k+
% (\gamma_{N-1}-\gamma_{N})V_{N}(b_{N-1})=0.
% \end{equation}
% The function 
% \[f_{N-1}(b):=(\gamma_{N-1}+k-1)b-\gamma_{N-1}k+
% (\gamma_{N-1}-\gamma_{N})V_{N}(b)\]
% is convex and satisfies $f(0)=-\gamma_{N-1}k$
% and $f(c_{N-1})>0$, so there exists a unique zero $b=b_{N-1}$, with $b_{N-1}\in (0,c_{N-1})$.
% As before, it is now straightforward to provide a verification argument showing that $\hat V_{N-1}=V_{N-1}$ (for further details, see 
% Proposition~\ref{generalverifdisc} below.

% We note that $b_{N-1}\leq b_N$, which was used in the derivation above. 

Next we treat the case $n=0,\ldots, N-1$ using induction. Assume that there are points $b_{n+1}\leq b_{n+2}\leq ...\leq b_{N}$ such that
\begin{equation}\label{Vm}
V_{m}(\pi)=\left\{\begin{array}{cl}
    A_{m}G_{m}(\pi) & \pi<b_{m}\\
    \pi-k+V_{m+1}(\pi) & \pi\geq b_{m},\end{array}\right. 
\end{equation} 
for $m=n+1,...,N$, where $V_{N+1}\equiv 0$.
Also assume that
\begin{equation}
\label{bm}
(\gamma_{m}+k-1)b_{m}-\gamma_{m}k+
(\gamma_{m}-\gamma_{m+1})V_{m+1}(b_{m})=0.
\end{equation}

\begin{remark}\label{rem}
Equation \eqref{bm} holds for $m=N$ with $b_N=c_N$, and for $m\leq N-1$ it is a consequence of the so-called smooth-fit condition $V'_m(b_m)=1+V'_{m+1}(b_m)$ (cf., for example, \cite{PS}) and the assumed monotonicity of the boundary. Indeed, the smooth-fit condition at $b_{m}$ gives the equation system 
\begin{equation}
        \left\{\begin{array}{l}
    A_{m}G_{m}( b_{m}) = b_{m}-k +V_{m+1}(b_{m}) \\
    A_{m}G'_{m}(b_{m}) =1+ V'_{m+1}(b_{m}) ,
    \end{array}\right.
\end{equation}
which yields
\begin{equation}\label{bndeqnm}
(\gamma_{m}+k-1)b_{m}-\gamma_{m}k=
b_{m}(1-b_{m})V'_{m+1}(b_{m})-(\gamma_{m}-b_{m})V_{m+1}(b_{m}).
\end{equation}
Now, if $b_{m}\leq b_{m+1}$, then 
\[V'_{m+1}(b_{m})=\frac{\gamma_{m+1}-b_{m}}{b_{m}(1-b_{m})}V_{m+1}(b_{m}),\]
and \eqref{bndeqnm} reduces to \eqref{bm}.
\end{remark}

We now provide conditions under which also $V_n$ has the form 
 \begin{equation}\label{bn}
    \hat V_{n}(\pi)=\left\{\begin{array}{cl}
    A_{n}G_{n}(\pi) & \pi<b_{n}\\
    \pi-k+V_{n+1}(\pi) & \pi\geq b_{n},\end{array}\right. \end{equation}
where the boundary point $b_n$ satisfies \eqref{bN-1} below (which is \eqref{bm} with $m=n$).
To do that, first note that if there exists a boundary point $b_n$ as in \eqref{bn}, then 
the smooth-fit condition reads
\begin{equation}
        \left\{\begin{array}{l}
    A_{n}G_{n}( b_{n}) = b_{n}-k +V_{n+1}(b_{n}) \\
    A_{n}G'_{n}(b_{n}) =1+ V'_{n+1}(b_{n}) ,
    \end{array}\right.
\end{equation}
which reduces (as in Remark~\ref{rem}) to 
\begin{equation}
\label{bN-1}
(\gamma_{n}+k-1)b_{n}-\gamma_{n}k+
(\gamma_{n}-\gamma_{n+1})V_{n+1}(b_{n})=0,
\end{equation}
provided $b_n\leq b_{n+1}$.
Denote
\begin{equation}\label{fn}
f_{n}(b):= (\gamma_{n}+k-1)b-\gamma_{n}k+ (\gamma_{n}-\gamma_{n+1})V_{n+1}(b),
\end{equation}
and let similarly 
\[f_{m}(b):= (\gamma_{m}+k-1)b-\gamma_{m}k+ (\gamma_{m}-\gamma_{m+1})V_{m+1}(b)\]
for $m=n+1,...,N$ so that $f_m(b_m)=0$.
Note that $f_{n}$ is convex, with $f_{n}(0)=-\gamma_{n} k<0$ and 
\[f_{n}(c_{n})=(\gamma_{n}-\gamma_{n+1})V_{n+1}(c_{n})>0.\]
Consequently,
there exists a unique $b_{n}\in(0,c_{n})$ such that $f_n(b_n)=0$, which defines $b_n$. 

\begin{remark}
We emphasize that the derived form of $f_n$ (as given in \eqref{fn}) uses that $b_n\leq b_{n+1}$; in particular, if the solution $b_n$ of $f_n(b_n)=0$ satisfies $b_n>b_{n+1}$, then the smooth-fit condition at $b_n$ is not guaranteed.
We also note that monotonicity of the boundary (i.e., $b_n\leq b_{n+1}$) is equivalent to 
$f_{n}(b_{n+1})\geq 0$.
\end{remark}

% \begin{proposition}\label{discpropmonot}
%     Assume $f_{n}(b_{n+1})\geq 0$. Then the solution $b_n:=b$ of the equation $f_{n}(b)=0$ satisfies $b_{n}\leq b_{n+1}$.
%     \end{proposition}

\begin{proposition}\label{generalverifdisc}
    Assume that $b_n\leq b_{n+1}$.
    % the function $f_n$ in \eqref{fn} satisfies $f_{n}(b_{n+1})\geq 0$. 
    Then, $\hat V_{n} = V_{n}$.
\end{proposition}

\begin{proof}
For the verification of $\hat V_{n}= V_{n}$, we first check that 
\begin{equation}\label{cond1_disc}
\hat V_{n}(\pi)\geq V_{n+1}(\pi) + \pi-k.
\end{equation}
Equation \eqref{cond1_disc} holds automatically (with equality) for $\pi\geq b_{n}$. 
To see that it holds also below $b_n$, assume that 
$\hat V_{n}(\pi_0)< V_{n+1}(\pi_0) + \pi_0-k$ for some $\pi_0<b_{n}$. We then have that
\begin{align*}
 \hat V_{n}'(\pi_0)-V_{n+1}'(\pi_0)-1 &= \frac{\gamma_{n}-\pi_0}{\pi_0(1-\pi_0)}\hat V_{n}(\pi_0)- \frac{\gamma_{n+1}-\pi_0}{\pi_0(1-\pi_0)}V_{n+1}(\pi_0)-1\\
 &< \frac{\gamma_{n}-\gamma_{n+1}}{\pi_0(1-\pi_0)}V_{n+1}(\pi_0)+ \frac{(\gamma_{n}-\pi_0)(\pi_0-k)-\pi_0(1-\pi_0)}{\pi_0(1-\pi_0)}\\
 &= \frac{1}{\pi_0(1-\pi_0)}f_{n}(\pi_0)\leq 0,
\end{align*}
where the last inequality follows from $\pi_0\leq b_{n}$. Thus, if $\hat V_{n}(\pi_0)< V_{n+1}(\pi_0) + \pi_0-k$ at some point $\pi_0<b_{n}$, then $\hat V_{n}(\pi)- V_{n+1}(\pi) - (\pi-k)$ is decreasing for $\pi\in[\pi_0,b_{n}]$, which contradicts the relation $\hat V_{n}(b_{n})= V_{n+1}(b_{n}) + b_{n}-k$; consequently, \eqref{cond1_disc} holds.

To complete a verification argument, we also need 
\begin{equation}
    \label{cond2_disc}
 \mathcal L_{n}\hat V_{n}:=   \frac{\rho_{n}^2}{2}\pi^2(1-\pi)^2\hat V_{n}''-r\hat V_{n}\leq 0
\end{equation}
for $\pi\not\in \{b_{n},b_{n+1}\}$. 
The inequality \eqref{cond2_disc} holds automatically (with equality) for $\pi< b_{n}$, so we only need to check it above $b_{n}$. 
For $\pi>b_{n}$ we have $\hat V_{n}=V_{n+1}+\pi-k$, so
\begin{align*} 
\mathcal L_{n}\hat V_{n}&=
\tfrac{\rho_{n}^2}{2}\pi^2(1-\pi)^2 V_{n+1}''-r V_{n+1}-r(\pi-k)\\
&=  \tfrac{\rho_{n}^2}{\rho_{n+1}^2}\mathcal L_{n+1} V_{n+1}-r\Big(1-\tfrac{\rho_{n}^2}{\rho^2_{n+1}}\Big)V_{n+1}-r(\pi-k)\\
&\leq  -r\Big(1-\tfrac{\rho_{n}^2}{\rho^2_{n+1}}\Big)V_{n+1}-r(\pi-k),
\end{align*}
provided $\pi\not= b_{n+1}$ (and where $\mathcal L_{n+1}V_{n+1}:=\frac{\rho_{n+1}^2}{2}\pi^2(1-\pi)^2V_{n+1}''-rV_{n+1}\leq 0$). 
Since both $V_{n+1}$ and $\pi\mapsto \pi-k$ are increasing, we note that
\begin{align*} 
\mathcal L_{n}\hat V_{n}(\pi) &\leq -r\bigg(\Big(1-\tfrac{\rho_{n}^2}{\rho^2_{n+1}}\Big)V_{n+1}(b_{n})+b_{n}-k\bigg)\\
&= -r\Big(\tfrac{\gamma_{n}+\gamma_{n+1}-1}{\gamma_{n}^2-\gamma_{n}}(\gamma_{n}-\gamma_{n+1}) V_{n+1}(b_{n})+b_{n}-k\Big).
\end{align*}
Using $f_{n}(b_{n})=0$, we have
\begin{align*}
(\gamma_{n}-\gamma_{n+1}) V_{n+1}(b_{n}) &=
\gamma_{n}k-(\gamma_{n}+k-1)b_{n}
\end{align*}
and so
\begin{align*} 
\mathcal L_{n}V_{n}(\pi) &\leq
-r\left(\frac{\gamma_{n+1}(k-b_{n})}{\gamma_{n}-1}  +\frac{\gamma_{n}+\gamma_{n+1}-1}{\gamma_{n}^2-\gamma_{n}}(1-k)b_{n} \right)\\
&\leq -r\left(\frac{\gamma_{n+1}k-(\gamma_{n+1}+k-1)b_{n}}{\gamma_{n}-1} \right)\leq 0,
\end{align*}
since $b_{n}\leq c_{n}\leq c_{n+1}$. Consequently, \eqref{cond2_disc} holds for all $\pi\not\in\{b_{n},b_{n+1}\}$. Using \eqref{cond1_disc} and \eqref{cond2_disc}, a standard verification procedure shows that $\hat V_{n} \equiv V_{n}$. 
%and that $\tau_{n}$ is an optimal stopping time.
\end{proof}

Proposition~\ref{generalverifdisc} completes the inductive construction and verification of the value function $V_n$. As remarked above, however, the construction depends on the assumption $b_n\leq b_{n+1}$. In the next subsection we provide conditions under which the boundary is indeed monotone.

\subsection{Monotonicity of the boundary}
First note that 
\[f_{N-1}(c_{N-1})=(\gamma_{N-1}-\gamma_N)V_N(c_{N-1})>0,\] 
so $b_{N-1}\in (0,c_{N-1})$. Since $c_{N-1}<c_N=b_N$, we automatically have $b_{N-1}\leq b_N$.

Now, for $n\in\{0,...,N-2\}$, assume that $b_{n+1}\leq...\leq b_N$ have been found such that \eqref{Vm} and \eqref{bm} hold for $m=n+1,...,N$.
We then have
\begin{equation}\label{discvn+1}
    V_{n+1} (b_{n+1}) = b_{n+1}-k+V_{n+2}(b_{n+1}),
\end{equation}
and from $f_{n+1}(b_{n+1})=0$ we get
\begin{equation}\label{discvn+2}
    V_{n+2}(b_{n+1})= \frac{\gamma_{n+1}k-(\gamma_{n+1}+k-1)b_{n+1}}{\gamma_{n+1}-\gamma_{n+2}}.
\end{equation}
From the identities \eqref{discvn+1} and 
\eqref{discvn+2}, we thus obtain
\begin{equation}\label{discfnbn+1}
f_{n}(b_{n+1}) =2(b_{n+1}-k)(\gamma_{n}-\gamma_{n+1})+
\frac{(\gamma_{n+1}k-(\gamma_{n+1}+k-1)b_{n+1})(\gamma_{n}-2\gamma_{n+1}+\gamma_{n+2})}{\gamma_{n+1}-\gamma_{n+2}}.
\end{equation}

\begin{proposition}
    If
    \begin{equation}\label{discmonotone_cond1}
        2(\gamma_{n}-\gamma_{n+1})(\gamma_{n+1}-\gamma_{n+2})-(\gamma_{n}-2\gamma_{n+1}+\gamma_{n+2})\gamma_{n+1} \leq 0,
    \end{equation}
  % for all $n=0,\ldots,N-2$,
    then $b_{n}\leq b_{n+1}$.
  \end{proposition}
  
  \begin{proof}
  Define 
\begin{equation}\label{Hn}
H_n(b):=2(b-k)(\gamma_{n}-\gamma_{n+1})+
\frac{(\gamma_{n+1}k-(\gamma_{n+1}+k-1)b)(\gamma_{n}-2\gamma_{n+1}+\gamma_{n+2})}{\gamma_{n+1}-\gamma_{n+2}}
\end{equation}
so that $H_n(b_{n+1})=f_n(b_{n+1})$.
Then
\[
H_{n}(c_{n+1}) = 2(c_{n+1}-k)(\gamma_{n}-\gamma_{n+1}) >0 , 
\]
 and by ~\eqref{discmonotone_cond1} we have 
\[H_{n}(0) = -\frac{k}{\gamma_{n+1}-\gamma_{n+2}}\left((\gamma_{n}-\gamma_{n+1})(\gamma_{n+1}-\gamma_{n+2})-(\gamma_{n}-2\gamma_{n+1}+\gamma_{n+2})\gamma_{n+1}\right) \geq 0.
        \]
    Since $H_n$ is affine, it follows that $H_{n}(b)> 0$ for all $b\in(0,c_{n+1})$. Thus $f_n(b_{n+1})\geq 0$, and the result follows.
    \end{proof}

% We next study the case when 
% \begin{equation}\label{case2}
% 2(\gamma_{n}-\gamma_{n+1})(\gamma_{n+1}-\gamma_{n+2})-(\gamma_{n}-2\gamma_{n+1}+\gamma_{n+2})\gamma_{n+1} > 0
% \end{equation}
% for $n=0,...,N-2$. 
% Note that \eqref{case2} implies that
% \[2(\gamma_{n}-\gamma_{n+1})(\gamma_{n+1}-\gamma_{n+2})-(\gamma_{n}-2\gamma_{n+1}+\gamma_{n+2})(\gamma_{n+1}+k-1) > 0,\]
% so the function $H_n$ in \eqref{Hn}
% is increasing.
% % Also, note that $f_n(b_{n+1})=H_n(b_{n+1})$, cf. \eqref{discfnbn+1}. 
% For $n=0,...,N-2$, let 
% \[
% B_{n} := \frac{2(\gamma_{n}-\gamma_{n+1})(\gamma_{n+1}-\gamma_{n+2})-(\gamma_{n}-2\gamma_{n+1}+\gamma_{n+2})\gamma_{n+1}}{2(\gamma_{n}-\gamma_{n+1})(\gamma_{n+1}-\gamma_{n+2})-(\gamma_{n+1}+k-1)(\gamma_{n}-2\gamma_{n+1}+\gamma_{n+2})}k,
% \]
% and note that $0<B_{n}<c_{n}$ and $H_{n}(B_{n})=0$.
% Finally, let 
% \[B_{N-1}:=\frac{(2\gamma_{N-1}-\gamma_N)k}{2\gamma_{N-1}-\gamma_N+k-1}.\]

% \begin{proposition}
% Assume that \eqref{case2} holds, 
% and that $B_{n}$ is increasing in $n$ for $n=0,\ldots,N-1$. Then $b_{n}\leq b_{n+1}$, $n=0,...,N-1$.
% \end{proposition}
    
% \begin{proof}
% Since $H_{n}(b)$ is increasing in $b$, we have
% \[0 = H_{n}(B_{n}) \leq H_{n}(B_{n+1}) \leq H_{n}(b_{n+1})=(\gamma_{n+1}-\gamma_{n+2})f_n(b_{n+1}), 
% \]
% where the first inequality comes from the assumption that $B_n$ is increasing, and the second inequality follows from the assumption $B_{n+1}\leq b_{n+1}$. 
% It follows that $f_{n}(b_{n+1})\geq 0$, from which the result follows.
%     \end{proof}

\begin{remark}
Note that the condition \eqref{discmonotone_cond1} 
is a discrete version of \eqref{cond1}. 
%Similarly, \eqref{case2} corresponds to condition \eqref{cond2}. 
\end{remark}

Figure \ref{fig:discrete_case} illustrates the set of the optimal boundaries $b_n$, $n=0,\ldots,N$ for $\gamma_n$ that is a discrete version of the specification used in Figure~\ref{fig:b<k}.

\begin{figure}[H]
\centering
\includegraphics[scale=0.6]{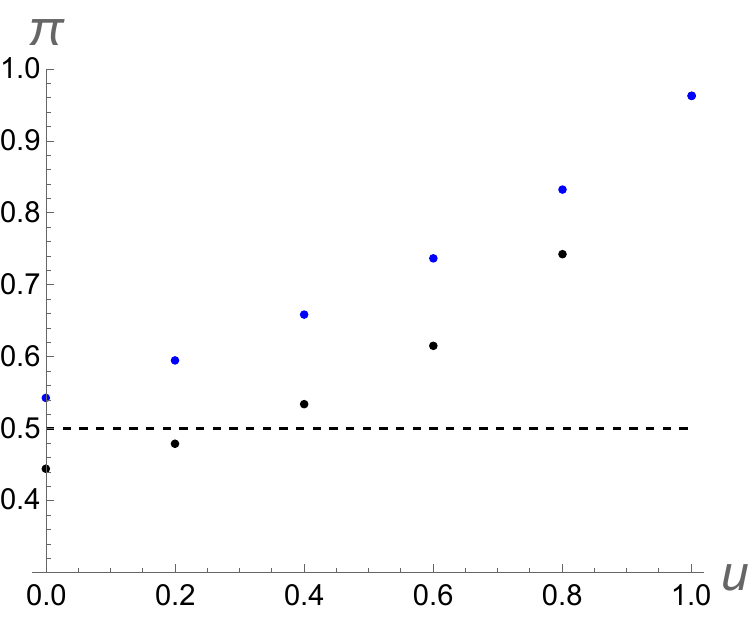}
\caption{Optimal boundaries $b_n$ (black dots), points $c_n$ (blue dots), and the threshold $k$ (black dashed) with $\gamma_n = 1.25/\left(\frac{n}{5} + 0.2\right)$. Remaining parameters are $k=0.5$  and $N=5$.
Note that $\gamma_n$ satisfies condition \eqref{discmonotone_cond1}.} 
\label{fig:discrete_case}
\end{figure}

% \begin{remark}
%      The function~\eqref{discfnbn+1} corresponds to the ODE~\eqref{F}. Indeed, using $f_{n}(b_n)=0$ ... 
%  %   Moreover, we emphasize that the discrete problem is well-defined if, for some $n=0,\ldots N-1$, the equation $f_{n}(b)=0$ produces a result $b:=b_{n}$, for which $b_{n}\leq b_{n+1}$. We assumed both of these conditions whescon deriving the function~\eqref{discfnbn+1}, and in the following passage we study conditions under which they hold.
% \end{remark}

% \begin{assumption}\label{assdiscr}
% For some $n$, $U = (U_n,\ldots U_1)$ is a regular partition of the interval $[0,1]$ for which
% \[
% \bP\left(U_t \in \left\{0, \frac{1}{n}, \frac{2}{n},\ldots, 1\right\}\right) = 1
% \]
% for all $t\in(0,1)$, and
% \[
% 0=U_{n+1}<U_n<U_{n-1}<\dots<U_2<U_1=1.
% \]
% Correspondingly, $\tau = (0=\tau_{n+1},\tau_n,\ldots,\tau_1)$, so that the control $U_n$ is selected at time $\tau_n$.

\bibliography{bibfile}{}
	\bibliographystyle{abbrv}

\end{document}